\documentclass[a4paper]{article}
\usepackage{amssymb}
\usepackage{amsmath,amsfonts,amsthm}
\usepackage{xcolor}

\newtheorem{theorem}{Theorem}[section]
\newtheorem{lemma}[theorem]{Lemma}
\newtheorem{proposition}[theorem]{Proposition}
\newtheorem{corollary}[theorem]{Corollary}

\usepackage{fullpage}
\allowdisplaybreaks

\begin{document}
	
	\numberwithin{equation}{section}
	
	\title{Ends of shrinking gradient $\rho$-Einstein solitons}
	
	\author{V. Borges\footnote{Universidade Federal do Pará, Faculdade de Matemática, 66075-110, Belém-PA, Brazil, valterborges@ufpa.br.} \quad H. A. Rosero-Garc\'ia \footnote{ Universidade de Bras\'{\i}lia,
			Department of Mathematics,
			70910-900, Bras\'{\i}lia-DF, Brazil, hector.garcia@unb.br. Partially supported by CAPES and CNPq.}
		\quad J. P. dos Santos \footnote{Universidade de Bras\'{\i}lia,
			Department of Mathematics,
			70910-900, Bras\'{\i}lia-DF, Brazil, joaopsantos@unb.br. Partially supported by CNPq 315614/2021-8.} 
	}
	
	\date{}
	
	\maketitle{}
	
	%\begin{abstract}
	%	We prove that all ends of a gradient shrinking $\rho$-Einstein soliton are $\varphi$-non-parabolic if $\rho$ is nonnegative, provided its scalar curvature is bounded and nonnegative, where $\varphi$ is a negative multiple of the potential function. We also show these solitons are connected at infinity for certain values of $\rho$, if its scalar curvature is suitably bounded.  
	%\end{abstract}
	
	% \begin{abstract}
		% We prove that all ends of a gradient shrinking $\rho$-Einstein soliton are $\varphi$-non-parabolic if $\rho$ is nonnegative, provided its scalar curvature is bounded and nonnegative, where $\varphi$ is a negative multiple of the potential function. We also show these solitons are connected at infinity for $\rho \in\left[0,1/2(n-1)\right)$, if $n\geq4$ and its scalar curvature is suitably bounded.
		% \end{abstract}
	
	\begin{abstract}
		We prove that all ends of a gradient shrinking $\rho$-Einstein soliton are $\varphi$-non-parabolic, provided $\rho$ is nonnegative and the soliton has bounded and nonnegative scalar curvature, where the weight $\varphi$ is a negative multiple of the potential function. We also show these solitons are connected at infinity for $\rho \in\left[0,1/2(n-1)\right)$, $n\geq4$, and a suitable bound for the scalar curvature.
		%and its scalar curvature is suitably bounded. \textcolor{blue}{Algumas alternativas para a ultima linha: (a) ... for $\rho \in\left[0,1/2(n-1)\right)$, and a suitable bound for the scalar curvature. \\(b) ... for $\rho \in\left[0,1/2(n-1)\right)$, and scalar curvature suitably bounded.}
	\end{abstract}
	
	\vspace{0.2cm} 
	\noindent\emph{2020 Mathematics Subject Classification} : 
	35Q51, % soliton equation
	53C20, % Global Riem geom
	53C21, % Methods of Riem. geom. including PDE methods
	53C25\\
	\emph{Keywords}: $\rho$-Einstein solitons, Ends, Parabolicity.% \textcolor{blue}{eu trocaria a ultima por ``parabolicity'' ou ``non-parabolicity''}

	\section{Introduction and Main Results}
	
	Given $\rho\in\mathbb{R}$, we say that a Riemannian manifold $(M^n,g)$ is a {\it gradient} $\rho$-{\it Einstein soliton} if there are $f\in C^{\infty}(M)$, called {\it potential function}, and $\lambda\in\mathbb{R}$ satisfying
	\begin{align}\label{genfundeq}
		Ric+\nabla\nabla f=\left(\rho R+\lambda\right)g.
	\end{align}
	Here, $\nabla\nabla f$ is the Hessian of $f$, $R$ is the scalar curvature, and $Ric$ is the Ricci tensor of $M$. The soliton is called {\it shrinking}, {\it steady} or {\it expanding}, provided $\lambda$ is positive, zero, or negative, respectively. In this case we use the notation $(M^n,g,f,\lambda)$. Notice that when $\rho=0$, these are the {\it gradient Ricci solitons}.
	%Gradient $\rho$-Einstein solitons may arise as self-similar solutions of the Ricci-Bourguignon flow
	
	%	Gradient $\rho$-Einstein solitons may generate self-similar solutions of the Ricci-Bourguignon flow \cite{catino1}. This family of geometric flows first appeared in \cite{bourg}, and have been investigated in \cite{catino}, for general $\rho$, and in \cite{hamilton1}, for $\rho=0$. The latter corresponds to the well known {\it Ricci flow}, introduced by Hamilton. If $\lambda\geq0$, these solutions are also ancient, in the sense they are defined for all negative times. Ancient solutions are important because they model singular regions, an important step in surgery procedures. Such families of metrics have plenty of special properties \cite{brgs2,catino2,pdaska}. For this reason, the investigation of topological, geometrical and analytical properties of these manifolds becomes an important aspect of the theory.
	
	Gradient $\rho$-Einstein solitons may generate self-similar solutions of the Ricci-Bourguignon flow \cite{catino1}, which is a family of geometric flows, studied in \cite{brgs2,bourg,catino}, for instance. A notable case occurs when $\rho=0$, once the flow corresponds to the well-known Ricci flow, introduced by Hamilton \cite{hamilton1}. If $\lambda>0$, self-similar solutions are also ancient solutions, in the sense they are defined for all negative times. Ancient solutions are crucial in modeling singular regions that may arise during the flow. Such families of metrics have plenty of special properties, as shown in \cite{brgs2,catino2,pdaska}. The connection with flows makes the investigation of the topological, geometrical, and analytical properties of these manifolds an important subject.
	
	%\textcolor{red}{The goal of this paper is} \textcolor{blue}{
		This paper aims to study the topology at infinity of shrinking gradient $\rho$-Einstein solitons, i.e., the ends structure of such manifolds. Recall that an {\it end} with respect to a compact, smooth subset of $M$ is an unbounded connected component of its complement, and the quantity of such components is the number of ends with respect to this compact subset. If this number	has an upper bound, no matter how big the compact set is, we say $M$ has finitely many ends and call the smallest upper bound $n_{0}=n_{0}(M)$ the {\it number of ends} of $M$. We also say $M$ has $n_{0}$ ends. If $n_{0}=1$, $M$ is said to be {\it connected at infinity}.
		
		Techniques for counting the ends of Riemannian manifolds were developed by Li and Tam in \cite{pLi0} (see also \cite{pLi,endharm}). A variation of this technique has been used by Munteanu and Wang to investigate the topology at infinity of a smooth metric measure spaces, see \cite{MuSe}-\cite{MuWa5}. This theory has successfully been used to show that gradient Ricci solitons are connected at infinity under quite general assumptions. For instance, when $\lambda<0$ and $R\geq-(n-2)\lambda$ in \cite{MuWa2}, when $\lambda=0$ in \cite{MuWa1}, and when $\lambda>0$ and $R\leq\frac{2n}{3}\lambda$ in \cite{MuWa5}. We observe that if we remove the assumption on the scalar curvature, the last result is false, once $\mathbb{R}\times\mathbb{S}^{n-1}$ is a shrinking Ricci soliton with $2$ ends \cite{QW}. Another notable use of this theory can be found in \cite{MuWa4}, where the authors show that Kähler shrinking Ricci solitons are connected at infinity.
		
		In order to count the number of ends of complete gradient shrinking $\rho$-Einstein solitons following the techniques of \cite{MuWa4,MuWa5}, it is convenient to regard $M^n$ as smooth metric measure space $(M^n,g,e^{-\varphi}dV)$, whose weight is given by
		\begin{equation}
			\varphi=-af,
		\end{equation}
		where $a > 0$ is a fixed constant. In this setting, an end $E$ is called $\varphi$-{\it non-parabolic} if the $\varphi$-Laplacian, $\Delta_{\varphi}(\cdot)=\Delta(\cdot)-\left\langle\nabla(\cdot),\nabla\varphi\right\rangle$, has a positive Green's function on $E$ satisfying the Neumann boundary condition. Otherwise, $E$ is called $\varphi$-{\it parabolic}.
		
		An important result established in \cite{MuWa4} asserts that complete gradient Ricci solitons have no $\varphi$-parabolic ends. This is a key step in the proof of connectedness at infinity mentioned above. The first result of this paper extends this important analytical description to the ends of a gradient shrinking $\rho$-Einstein solitons, under suitable assumptions:
		
		\begin{theorem}\label{nphinpendthm} 
			Let $(M^n, g, f,\lambda)$, $n\geq3$, be a complete gradient shrinking $\rho$-Einstein soliton with $\rho\geq0$ and scalar curvature satisfying $0\leq R\leq K$, for some positive constant $K$. Then all ends of $M$ are $\varphi$-non-parabolic.
		\end{theorem}
		
		Under the assumptions on the scalar curvature and $\rho$, demanded by the theorem above, it was proved in \cite{catino1} that the potential function grows quadratically. This is an important tool for the proof of Theorem \ref{nphinpendthm}. For Ricci solitons, this growth was obtained in \cite{caozhou}, and requires no assumptions on the scalar curvature, showing the bounds on $R$ are unnecessary \cite{MuWa4}. As for the Schouten solitons, i.e., $\rho=1/2(n-1)$, the scalar curvature is always non-negative and bounded \cite{brgs}. Therefore, we can state the following:

		\begin{corollary}\label{coro1} 
			Let $(M^n, g, f,\lambda)$, $n\geq3$, be a complete gradient shrinking Ricci or Schouten soliton. Then all ends of $M$ are $\varphi$-non-parabolic.
		\end{corollary}
		
		We now address the problem of counting the ends of $\rho$-Einstein solitons. It would be natural to try to generalize the result concerning Kähler Ricci solitons to Kähler $\rho$-Einstein solitons, obtained in \cite{MuWa4}. However, as a consequence of the results of Derdzinski and Maschler \cite{dema,ma}, such $\rho$-Einstein solitons must have constant scalar curvature, which reduces the analysis to Ricci solitons, where the result is already established. We thus follow the strategy of \cite{MuWa5} and consider a suitable bound on the scalar curvature, proving the following on the number of ends of $M$:
		
		\begin{theorem}\label{thm1endscl} Let $(M^n, g, f,\lambda)$, $n\geq4$, be a complete gradient shrinking $\rho$-Einstein soliton with $\rho \in\left[0,1/2(n-1)\right)$ and
			\begin{align}\label{estcond}
				0\leq R\leq\frac{2n(n-3)\lambda}{3(n-3)-(3n^2-12n+5)\rho}.
			\end{align}
			Then $M$ is connected at infinity.
		\end{theorem}
		
		The range $[0,1/2(n-1))$ in which the parameter $\rho$ lies in the theorem above was also considered in \cite{catino1}, where the authors proved rigidity of complete shrinking $\rho$-Einstein solitons under appropriate bounds on the Ricci and sectional curvatures. In that paper, this interval is required to show that the potential function is convex at infinity (see Proposition 4.2 of \cite{catino1}). Our reason is quite different. In fact, for the proof of Theorem \ref{thm1endscl}, we assume $\rho\geq0$ to get rid of the $\varphi$-parabolic ends. This follows from Theorem \ref{nphinpendthm}, once by \eqref{estcond} the scalar curvature is bounded. We then need to choose parameters in a certain inequality (see Proposition \ref{mainineqshr}) to exclude the possibility of more than one $\varphi$-non-parabolic. This is effective when the inverse of $1-2(n-1)\rho$ is positive, and since the bound on $R$ is as in \eqref{estcond}.
		
		A consequence of Theorem \ref{thm1endscl} is the impossibility of constructing a $\rho$-Einstein soliton through connected sums of knew ones without destroying \eqref{estcond}, when $\rho \in\left[0,1/2(n-1)\right)$. As described previously, the first inequality is not necessary when $\rho=0$. In fact, we already know $R>0$ on a gradient shrinking Ricci soliton with non-zero Ricci tensor \cite{chen,pigrimse}. It is reasonable to believe that the non-negativity of $R$ can be removed from Theorem \ref{thm1endscl}. This is motivated by the fact that compact ancient Ricci-Bourguignon flows with $\rho\leq 1/2(n-1)$ have non-negative scalar curvature.
		
		Theorems \ref{nphinpendthm} and \ref{thm1endscl} were motivated by the recent work on Ricci solitons by Munteanu and Wang \cite{MuWa4,MuWa5}. Besides the techniques for counting ends, mentioned above, these works rely on some identities at the disposal of Ricci solitons. One of these identities is the Hamilton identity \cite{hamilton2}, which reads as
		$$R+\vert\nabla f\vert^2-2\lambda=c_{0},$$
		where $c_{0}$ is a constant. Using this identity, it is possible to derive a variety of equations \cite{pointofview} which are essential to the theory. In the context of analyzing the topology at infinity, the Hamilton identity provides barrier functions to ends and suitable Poincaré inequalities \cite{MuWa4,MuWa5}. It is worthwhile mentioning that bounds to the scalar curvature of Ricci solitons follow from this identity combined with maximum principles \cite{chowluyang}.
		
		Unfortunately, a similar identity is not yet available for $\rho$-Einstein solitons. We overcome this obstacle by using two inequalities connecting the potential function and its gradient, obtained by Catino et. al \cite{catino1}, which is accessible at our conditions (see also \cite{brgs} for Schouten solitons, where no conditions are required). Since we are using a couple of inequalities rather than an equality, the proofs are not identical to those in \cite{MuWa4,MuWa5}. Furthermore, the $\rho$-Einstein soliton equation \eqref{genfundeq} carries a scalar curvature term not present in the Ricci soliton equation, which brings further computational challenges.
		
		This paper is organized as follows. In Section 2, we recall basic identities of $\rho$-Einstein solitons and state important results on smooth metric measure spaces
		%that can be found in \cite{catino,catino1,MuSe,MuWa1,MuWa2,MuWa3,MuWa4,MuWa5}, and  
		which will be used in the proofs of our results. In Section \ref{phinonpar}, we prove Theorem \ref{nphinpendthm}. In Section \ref{onephinpend}, assuming the existence of two $\varphi$-non-parabolic ends, we establish a two-parameter inequality relating the resulting $\varphi$-harmonic function given by the theory of Li and Tam with the geometry of $M$. In Section \ref{lastsec}, we choose the parameters appropriately to prove Theorem \ref{thm1endscl}.

		\section{Preliminaries}\label{preliminariesresults}
		In this section, we recall certain results used in the proof of the main results of this paper. The first of them lists some consequences of equation \eqref{genfundeq}, which are central in the study of gradient $\rho$-Einstein solitons, obtained in \cite{catino}.
		
		\begin{proposition}[\cite{catino}]\label{catino}
			Let $(M^n,g,f,\lambda)$, $n\geq3$, be a gradient $\rho$-Einstein soliton. Then, the following identities hold
			\begin{eqnarray}
				&\Delta f=(n\rho-1)R+n\lambda,\label{cat1}\\
				&(1-2(n-1)\rho)\nabla R=2Ric(\nabla f),\label{cat2}\\
				&(1-2(n-1)\rho)\Delta R=\langle\nabla R,\nabla f\rangle + 2(\rho R^2-\vert Ric\vert^2+\lambda R)\label{cat3}.
			\end{eqnarray}
		\end{proposition}
		
		Assuming the scalar curvature of a $\rho$-Einstein soliton is non-negative and bounded, Catino et. al used Proposition \ref{catino} to prove in \cite{catino1} the following estimate for the gradient of the potential function.
		
		\begin{proposition}[\cite{catino1}]\label{cor33RGES} Let $(M^n,g,f,\lambda)$ be a gradient shrinking $\rho$-Einstein soliton with $\rho>0$, scalar curvature $R\geq0$ and such that $|R|<K$ for some positive constant $K$. Then, either $f$ is constant or there exist positive real constants $\alpha, \beta, \epsilon, \delta$ such that
			$$\alpha f(r)-\beta\leq \vert\nabla f\vert^2(r)\leq \epsilon f(r)+\delta,$$
			where $r$ is the distance to a connected component $\Sigma_0\subset M$ of some regular level set of $f$.
		\end{proposition}
		
		As pointed out in \cite{catino1}, the constants $\alpha, \beta, \epsilon$ and $\delta$ are possibly depending on the regular level $\Sigma_{0}$. A particular consequence of this proposition is that the potential function satisfies $\vert\nabla f\vert^2\leq \epsilon f+\delta$. Once we also have 
		$$Ric_{f}\geq\lambda g,$$
		with $\lambda>0$, we can apply the following result of \cite{MuWa3} to obtain the asymptotic behavior of $f$.
		
		\begin{proposition}[\cite{MuWa3}]\label{prop42ineqfd}
			Let $(M^{n}, g, e^{-f}dV)$ be a complete smooth metric measure space. Assume $\emph{Ric}_f\geq\frac{1}{2}$ and $|\nabla f|^2\leq f$. Then there exists a constant $a>0$ such that
			$$\frac{1}{4}(d(x,x_0)-a)^2\leq f(x)\leq\frac{1}{4}(d(x,x_0)+a)^2,$$
			for any $x\in M$ and $d(x,x_0)\geq r_0$. The constants $a$ and $r_0$ depend only on $n$ and $f(x_0)$.
		\end{proposition}

		Shrinking $\rho$-Einstein solitons under our assumptions have many other interesting properties, regarded as smooth metric measure spaces. See \cite{MuWa3}.
		
		We now present useful properties of $\varphi$-parabolic and $\varphi$-non-parabolic ends. Recall that $u$ is $\varphi$-harmonic function if $\Delta_\varphi u=\Delta u-\langle \nabla\varphi,\nabla u\rangle=0$.

		%	\begin{proposition}\label{thmallendfpara}
			%		If there is a positive $\varphi$-harmonic function $u$ on $M$ so that 
			%		$$\lim_{y\to\infty}u(y)=0,$$
			%		then all ends of $M$ are $\varphi$-non-parabolic.
			%	\end{proposition}
		
		\begin{proposition}\label{thmfparachar}
			Suppose $M$ has at least two ends. An end $E$ of $M$ is $\varphi$-parabolic if and only if there is a positive $\varphi$-harmonic function $u$ defined on $E$ so that $u\geq1$, $u=1$ on $\partial E$ and 
			$$\lim_{y\to E(\infty)}u(y)=\infty,$$
			where $E(\infty)$ denotes the infinity of $E$.% The function $u$ is said to be a barrier function of $E$.
		\end{proposition}
		
		\begin{proposition}\label{defuvarphi}
			If $M$ has two $\varphi$-non-parabolic ends $E_{1}$ and $E_{2}$, then there is a $\varphi$-harmonic function $u$ satisfying $0<u< 1$ on $M$, $\inf_{E_1}u=0$, $\sup_{E_2}u=1$ and
			\begin{equation*}
				\int\limits_M|\nabla u|^2e^{-\varphi}<\infty.
			\end{equation*}
		\end{proposition}
		
		These propositions have been used in \cite{MuWa1,MuWa2,MuWa4,MuWa5} to investigate parabolicity issues regarding ends of gradient Ricci solitons, leading to a broad description of the topology at infinity of these manifolds. When the weight $\varphi$ is constant, Proposition \ref{thmfparachar}
		and Propostion \ref{defuvarphi} are related with Lemma 20.7 and Theorem 21.3 of \cite{pLi}, respectively. We are obviously interested on weights of the form $\varphi=-af$, where $f$ is the potential function, and $a$ is a constant.

		\section{$\varphi$-Non-parabolicity and ends}\label{phinonpar}
		
		In this section we prove Theorem \ref{nphinpendthm}, namely, that a shrinking $\rho$-Einstein soliton with $\rho\geq0$ and $0\leq R\leq K$ has only $\varphi$-non-parabolic ends. Following \cite{MuWa4,MuWa5}, we regard $M^n$ as smooth metric measure space $(M^n,g,e^{-\varphi}dV)$, whose weight is $\varphi=-af$, and $a>0$ is a fixed constant. The Bakry–Emery Ricci tensor associated with this weighted smooth metric measure space is given by $\mbox{Ric}_\varphi=\mbox{Ric}+\nabla\nabla\varphi$. Once $M$ is a $\rho$-Einstein soliton, we have
		$$\mbox{Ric}_\varphi=(\rho R+\lambda)g-(a+1)\nabla\nabla f.$$

		For $\rho=0$, Theorem \ref{nphinpendthm} was obtained by Munteanu and Wang in \cite{MuWa4} as an ingredient to prove the following result.
		\begin{theorem}[\cite{MuWa4}]\label{thmmw15}
			Let $(M^n,g,f,\lambda)$ be a gradient shrinking Kähler Ricci soliton. Then $(M^n,g)$ has only one end.
		\end{theorem}
		
		Since the proof of Theorem \ref{nphinpendthm} is based on Munteanu and Wang's proof of Theorem \ref{thmmw15} in \cite{MuWa4}, it is convenient to recall the main steps of that proof. In the description below, one is taking $\lambda=1/2$ and assuming $M$ is not connected at infinity. The proof can be divided into the following five steps:
		\begin{itemize}
			\item \textbf{Step 1.} By assuming the existence of a $\varphi$-parabolic end $E$, it is proved a $\varphi$-harmonic function $h$ satisfying the properties of Proposition \ref{thmfparachar} satisfies the following integral estimate
			\begin{equation}\label{step1ineqmw}
				\int\limits_{B_x(1)}|\nabla \ln h|^2\leq C_0e^{-\frac{a}{4}r(x)^2+cr(x)}.
			\end{equation}
			\item \textbf{Step 2.} Setting $v=\ln{h}$ and $\sigma=|\nabla v|^2$  with the aim of {transforming} (\ref{step1ineqmw}) into a point-wise inequality, an integral inequality in terms only of $\sigma$, $\nabla \sigma$ and $f$ is proved, namely
			\begin{equation}\label{step2ineq}
				\begin{split}
					({p}-{2})\int\limits_M\sigma^{p-2}|\nabla\sigma|^2\phi^2\leq&\ C(n)(a+1)^2p\int\limits_Mf\sigma^p\phi^2
					+C(n)(a+1)\int\limits_M\sigma^p|\nabla\phi|^2\\
				\end{split}
			\end{equation}
			for $p$ large enough, a constant $C(n)$ independent of $p$ and for any cut-off function $\phi$ with support on the unit ball $B_{x_0}(1)$, for a fixed point $x_0\in E$.
			\item \textbf{Step 3.} Combining (\ref{step2ineq}) with the Sobolev Inequality obtained in \cite{MuWa1}, Nash-Moser theory is applied to conclude $\sigma$ satisfies the mean value inequality
			$$\sigma(x)\leq Ce^{c(n)r(x)}\int\limits_{B_x(1)}\sigma,$$
			where $r(x)=d(x_0,x)$. 
			\item \textbf{Step 4.} The integral inequality found in {\it Step 3} is combined with \ref{step1ineqmw} to conclude $h$ must be bounded, which is a contradiction, implying that all ends of $M$ must be $\varphi$-non-parabolic.
			\item \textbf{Step 5.} The Kähler geometry of $M$ is used to show that the assumption of more than one $\varphi$-non-parabolic end leads to a contradiction, concluding $M$ must be connected at infinity. 
		\end{itemize}
		
		It is worth mentioning that being Kähler in Theorem \ref{thmmw15} is used only on the last step to conclude connectedness at infinity and that the first four steps are still valid for a Riemannian shrinking gradient Ricci soliton. Thus, the argument proves Theorem \ref{nphinpendthm} when $\rho=0$, without any assumption on the scalar curvature.
		
		Consequently, if Steps 1 to 4 are true for a shrinking gradient $\rho$-Einstein soliton with non-negative bounded scalar curvature, then all its ends are $\varphi$-non-parabolic. We notice that Step 1 and Step 2 are strongly based on the Ricci soliton structure. Thus, it is not obvious that these steps will hold true in any other configuration. We show that both steps are true under the conditions of Theorem \ref{nphinpendthm}. Starting with inequality \eqref{step2ineq}, Step 3 is guaranteed whenever a suitable Sobolev Inequality is available. By the results of \cite{MuWa1,MuWa3}, as observed in \cite{MuWa4}, this is the case when $(M^n,g,e^{-f}dV)$ is a smooth metric measure space with nonnegative Bakry-Emery Ricci curvature. This is exactly the case on a shrinking gradient $\rho$-Einstein soliton with bounded and nonnegative scalar curvature. Finally, step 4 can be performed once we have Step 1 and Step 3 established.
		
		\begin{proof}[Proof of Theorem \ref{nphinpendthm}]
			Suppose by contradiction that $E$ is a $\varphi$-parabolic end of $M$. Then, from Proposition \ref{thmfparachar}, there exists a proper $\varphi$-harmonic function $h$ on the end such that,
			%\textcolor{green}{Colocar resultado de Nakai nas preliminaries}
			\begin{equation}\label{hconditions}
				\begin{split}
					h\geq 1 \mbox{ on } E, \ h=1 \mbox{ on } \partial E,\ \lim\limits_{x\to E(\infty)} h(x)=\infty, \mbox{ and }\ \Delta_\varphi h=0.
				\end{split}
			\end{equation}
			Our goal is to show that (\ref{hconditions}) leads to a contradiction, which proves the theorem.
			
			For $t>1$ and $1<b<c$, we define the sets
			\begin{equation*}
				l(t):=\{x\in E: h(x)=t\}\ \ \ \text{and}\ \ \ L(b,c):=\{x\in E: b<h(x)<c\}.
			\end{equation*}
			From (\ref{hconditions}) we know $l(t)$ and $\overline{L(b,c)}$ are compact. Since on the level set $h=t$ we have the unit normal $\nu=\nabla h/|\nabla h|$ we get by the divergence theorem that
			\begin{equation*}
				\begin{split}
					0=\int\limits_{L(b,c)}(\Delta_\varphi h)e^{-\varphi} =\int\limits_{\partial L(b,c)}(\partial_\nu h)e^{-\varphi}=\int\limits_{\partial L(b,c)}\left\langle\frac{\nabla h}{|\nabla h|},\nabla h\right\rangle e^{-\varphi}=\int\limits_{l(c)}|\nabla h|e^{-\varphi}-\int\limits_{l(b)}|\nabla h|e^{-\varphi},
				\end{split}
			\end{equation*}
			which implies the weighted integral of $|\nabla h|$ over $l(t)$ is independent of $t$. Consequently,
			\begin{equation*}%\label{coareah}
				\begin{split}
					\int\limits_{E}|\nabla \ln h|^2e^{-\varphi}= \int\limits_{L(1,\infty)}\frac{|\nabla h|^2}{h^2}e^{-\varphi}= \int\limits_1^\infty\frac{1}{t^2}dt\int\limits_{l(t_0)}|\nabla h|e^{-\varphi}=C_0<\infty.
				\end{split}
			\end{equation*}
			where we have used the co-area formula. Now, for any $x\in E$ with $B_x(1)\subset E$, we have 
			\begin{equation}\label{minephibx1}
				\left(\min\limits_{x\in B_x(1)}e^{-\varphi}\right)\int\limits_{B_x(1)}|\nabla \ln h|^2e^{-\varphi}\leq\int\limits_{B_x(1)}|\nabla \ln h|^2e^{-\varphi}e^{-\varphi}\leq C_0.
			\end{equation}
			If $e^{-\varphi}$ attains its minimum over $\overline{B_x(1)}$ at $\Tilde{x}$, Proposition \ref{prop42ineqfd} assures that
			$$\min\limits_{x\in \overline{B_x(1)}}e^{-\varphi}=e^{af(\Tilde{x})}\geq e^{a(\frac{1}{4}r(\Tilde{x})-\alpha)^2}\geq e^{\frac{a}{16}r(\Tilde{x})^2-\frac{a\alpha}{2}r(\Tilde{x})}$$
			where $r(x)=d(x,x_0)$, for a fixed point $x_0\in M$. Once $r(x)+1\geq r(\Tilde{x})\geq r(x)-1$ for $x\in B_x(1)$, we have
			$$\min\limits_{x\in B_x(1)}e^{-\varphi}\geq e^{\frac{a}{16}(r(x)-1)^2-\frac{a\alpha}{2}(r(x)+1)},$$
			which together with (\ref{minephibx1}) implies
			\begin{equation}\label{ineqlnhexp}
				\int\limits_{B_x(1)}|\nabla \ln h|^2\leq C_0e^{-\frac{a}{16}(r(x)-1)^2+\frac{a\alpha}{2}(r(x)+1)}.
			\end{equation}
			{Notice equation (\ref{ineqlnhexp}) is analogous to that found in Step 1 of the proof of Theorem \ref{thmmw15}.}
			
			We now work on showing the function $\sigma:=|\nabla\ln{h}|^2$ satisfies an inequality analogous to (\ref{step2ineq}). Setting $v:=\ln{h}$,	then $\nabla v=\nabla h/h$, and, since $h$ is $\varphi$-harmonic,
			\begin{equation}\label{eq26lap}
				\begin{split}
					\Delta v=& -a\langle\nabla v,\nabla f\rangle -|\nabla v|^2.
				\end{split}
			\end{equation}
			By applying the Bochner formula on $v$, and the $\rho$-Einstein equation, we get
			\begin{equation}\label{eqboch0}
				\begin{split}
					\frac{1}{2}\Delta|\nabla v|^2%=&|\nabla\nabla v|^2+\langle\nabla\Delta v,\nabla v\rangle+ \mbox{Ric}(\nabla v,\nabla v)\\
					=&|\nabla\nabla v|^2-a\langle\nabla\left\langle\nabla v,\nabla f\rangle,\nabla v\right\rangle-\langle\nabla|\nabla v|^2,\nabla v\rangle+ (\rho R+\lambda)|\nabla v|^2-\nabla\nabla f(\nabla v,\nabla v).\\
				\end{split}
			\end{equation}
			Notice from Theorem \ref{cor33RGES} and Young's inequality,
			\begin{equation}\label{ineqbochv2}
				\begin{split}
					a\nabla\nabla v(\nabla f,\nabla v)%\leq&a|\nabla\nabla v||\nabla f||\nabla v|
					\leq\frac{1}{2}|\nabla\nabla v|^2+\frac{a^2}{2}|\nabla f|^2|\nabla v|^2\leq\frac{1}{2}|\nabla\nabla v|^2+\frac{a^2(\delta f+\epsilon)}{2}|\nabla v|^2.
				\end{split}
			\end{equation}
			On the other hand, from Schwarz inequality, (\ref{eq26lap}) and Young's inequality, 
			\begin{equation}\label{ineqbochv3}
				\begin{split}
					|\nabla\nabla v|^2\geq&\frac{(\Delta v)^2}{n}=\frac{1}{n}\left(a\langle\nabla v,\nabla f\rangle +|\nabla v|^2\right)^2=\frac{1}{n}\left(a^2\langle\nabla v,\nabla f\rangle^2+2a\langle\nabla v,\nabla f\rangle|\nabla v|^2 +|\nabla v|^4\right)\\
					%\geq&{\ }\frac{1}{n}\left(2a\langle\nabla v,\nabla f\rangle|\nabla v|^2 +|\nabla v|^4\right)    \geq\frac{1}{n}\left(|\nabla v|^4-2a|\nabla v||\nabla f||\nabla v|^2\right)\\
					\geq&{\ }\frac{1}{n}\left(|\nabla v|^4-\frac{1}{2}(2a|\nabla v||\nabla f|)^2-\frac{1}{2}|\nabla v|^4\right)=\frac{1}{n}\left(\frac{1}{2}|\nabla v|^4-2a^2|\nabla v|^2|\nabla f|^2\right)\\
					\geq&{\ }\frac{1}{2n}|\nabla v|^4-a^2|\nabla v|^2|\nabla f|^2\geq\frac{1}{2n}|\nabla v|^4-a^2(\delta f+\epsilon)|\nabla v|^2.    
				\end{split}
			\end{equation}
			By plugging (\ref{ineqbochv2}) and (\ref{ineqbochv3}) into (\ref{eqboch0}) we get
			\begin{equation}\label{eqboch1}
				\begin{split}
					\frac{1}{2}\Delta|\nabla v|^2
					%=&|\nabla\nabla v|^2-a\langle\nabla\left\langle\nabla v,\nabla f\rangle,\nabla v\right\rangle-\langle\nabla|\nabla v|^2,\nabla v\rangle+ (\rho R+\lambda)|\nabla v|^2-\nabla\nabla f(\nabla v,\nabla v)\\
					=&|\nabla\nabla v|^2-a\left(\nabla\nabla f(\nabla v,\nabla v)+\nabla\nabla v(\nabla f,\nabla v)\right)-\langle\nabla|\nabla v|^2,\nabla v\rangle \\
					&+ (\rho R+\lambda)|\nabla v|^2-\nabla\nabla f(\nabla v,\nabla v)\\
					\geq&|\nabla\nabla v|^2-a\nabla\nabla f(\nabla v,\nabla v)-\frac{1}{2}|\nabla\nabla v|^2-\frac{a^2(\delta f+\epsilon)}{2}|\nabla v|^2-\langle\nabla|\nabla v|^2,\nabla v\rangle \\
					&+ (\rho R+\lambda)|\nabla v|^2-\nabla\nabla f(\nabla v,\nabla v)\\
					=&\frac{1}{2}|\nabla\nabla v|^2-\langle\nabla|\nabla v|^2,\nabla v\rangle+ \left(\rho R+\lambda-\frac{a^2(\delta f+\epsilon)}{2}\right)|\nabla v|^2-(a+1)\nabla\nabla f(\nabla v,\nabla v)\\
					\geq&\frac{1}{2}\left(\frac{1}{2n}|\nabla v|^4-a^2(\delta f+\epsilon)|\nabla v|^2\right)-\langle\nabla|\nabla v|^2,\nabla v\rangle+\left(\rho R+\lambda-\frac{a^2(\delta f+\epsilon)}{2}\right)|\nabla v|^2\\
					&-(a+1)\nabla\nabla f(\nabla v,\nabla v)\\
					=&\frac{1}{4n}|\nabla v|^4-\langle\nabla|\nabla v|^2,\nabla v\rangle+\left(\rho R+\lambda-a^2(\delta f+\epsilon)\right)|\nabla v|^2-(a+1)\nabla\nabla f(\nabla v,\nabla v)
				\end{split}
			\end{equation}
			By making $\sigma=|\nabla v|^2$, inequality (\ref{eqboch1}) can be rewritten using Einstein summation for the term $\nabla\nabla f(\nabla v,\nabla v)$ as
			\begin{equation}\label{ineq27sigma}
				\sigma^2\leq4n\left(a^2(\delta f+\epsilon)-\rho R-\lambda\right)\sigma+4n\langle\nabla\sigma,\nabla v\rangle+4n(a+1)f_{ij}v_iv_j+2n\Delta\sigma.
			\end{equation}
			
			Consider $\phi$ a cut-off function over $M$ and {$p>0$ large enough depending only on $n$}. Then, multiplying inequality (\ref{ineq27sigma}) by $\sigma^{p-1}\phi^2$, and integrating over $M$ we get
			\begin{equation}\label{ineq28sigma}
				\begin{split}
					\int\limits_M\sigma^{p+1}\phi^2\leq&\ 4n\int\limits_M\left(\rho R+\lambda-a^2(\delta f+\epsilon)\right)\sigma^p\phi^2 +4n\int\limits_M\langle\nabla\sigma,\nabla v\rangle\sigma^{p-1}\phi^2\\ &+4n(a+1)\int\limits_Mf_{ij}v_iv_j\sigma^{p-1}\phi^2 +2n\int\limits_M\sigma^{p-1}(\Delta\sigma)\phi^2.
				\end{split}
			\end{equation}
			Now, we proceed to bound the right side of (\ref{ineq28sigma}) in terms of only $f$, $\sigma$, $\phi$ and $\nabla\phi$. 
			
			By noticing that
			$$\int\limits_M (\sigma^{p}\phi^2)\Delta v =-\int\limits_M\langle\nabla(\sigma^{p}\phi^2),\nabla v\rangle =-\int\limits_M\langle\nabla\sigma^{p},\nabla v\rangle\phi^2 -\int\limits_M\langle\nabla\phi^2,\nabla v\rangle\sigma^{p},$$
			we get that the integral on the second term on the right-hand side of (\ref{ineq28sigma}) can be bounded using Young's inequality, (\ref{eq26lap}) and (\ref{ineqbochv2}) as follows,
			\begin{equation}\label{ineq29sigma}
				\begin{split}
					\int\limits_M\langle\nabla\sigma,\nabla v\rangle\sigma^{p-1}\phi^2=& \frac{1}{p}\int\limits_M\langle\nabla\sigma^p,\nabla v\rangle\phi^2=
					-\frac{1}{p}\int\limits_M (\sigma^{p}\phi^2)\Delta v-\frac{1}{p}\int\limits_M\langle\nabla\phi^2,\nabla v\rangle\sigma^{p}\\
					=&\frac{1}{p}\int\limits_M \sigma^{p}\left(a\langle\nabla v,\nabla f\rangle +\sigma\right)\phi^2-\frac{1}{p}\int\limits_M\langle\nabla\phi^2,\nabla v\rangle\sigma^{p}\\
					\leq& \frac{a}{p}\int\limits_M \sigma^{p}\langle\nabla v,\nabla f\rangle \phi^2+\frac{1}{p}\int\limits_M \sigma^{p+1}\phi^2 -\frac{2}{p}\int\limits_M\phi\langle\nabla\phi,\nabla v\rangle\sigma^{p}\\
					\leq& \frac{a}{p}\int\limits_M \sigma^{p}|\nabla v||\nabla f| \phi^2+\frac{1}{p}\int\limits_M \sigma^{p+1}\phi^2 +\frac{2}{p}\int\limits_M|\nabla\phi||\nabla v|\sigma^{p}\phi\\
					\leq&\frac{a^2}{2p}\int\limits_M \sigma^{p}|\nabla f|^2 \phi^2 +\frac{1}{2p}\int\limits_M \sigma^{p+1}\phi^2 +\frac{1}{p}\int\limits_M \sigma^{p+1}\phi^2 +\frac{2}{p}\int\limits_M|\nabla\phi|^2\sigma^{p}\\ &+\frac{1}{2p}\int\limits_M\phi^2\sigma^{p+1}\\
					\leq&\frac{a^2}{2p}\int\limits_M \sigma^{p}(\delta f+\epsilon) \phi^2 +\frac{2}{p}\int\limits_M \sigma^{p+1}\phi^2 +\frac{2}{p}\int\limits_M|\nabla\phi|^2\sigma^{p}. 
				\end{split}
			\end{equation}
			Regarding the last term on the right side of (\ref{ineq28sigma}), by integrating by parts we get
			\begin{equation*}
				\begin{split}
					\int\limits_M\sigma^{p-1}(\Delta\sigma)\phi^2=-(p-1)\int\limits_M\sigma^{p-2}\langle\nabla\sigma,\nabla\sigma\rangle\phi^2-2\int\limits_M\sigma^{p-1}\langle\nabla\phi,\nabla\sigma\rangle\phi,
				\end{split}
			\end{equation*}
			second term of the right-hand side of the equation above is bounded by
			$$-2\int\limits_M\sigma^{p-1}\langle\nabla\phi,\nabla\sigma\rangle\phi\leq2\int\limits_M\sigma^{p-2}|\nabla\phi||\nabla\sigma|\sigma\phi\leq\int\limits_M\sigma^{p-2}\left(|\nabla\sigma|^2\phi^2+|\nabla\phi|^2\sigma^2\right).$$
			Thus, {for $p$ big enough} %($-(p-2)\leq\nicefrac{-p}{2}\Longleftrightarrow p\geq4$) 
			%\begin{equation}\label{ineq210sigma}
			%    \begin{split}
				%        \int\limits_M\sigma^{p-1}(\Delta\sigma)\phi^2\leq-(p-2)\int\limits_M\sigma^{p-2}|\nabla\sigma|^2\phi^2+\int\limits_M\sigma^{p}|\nabla\phi|^2.
				%    \end{split}
			%\end{equation}
			\begin{equation}\label{ineq210sigma}
				\begin{split}
					\int\limits_M\sigma^{p-1}(\Delta\sigma)\phi^2\leq&-(p-2)\int\limits_M\sigma^{p-2}|\nabla\sigma|^2\phi^2+\int\limits_M\sigma^{p}|\nabla\phi|^2\\
					\leq&-\frac{p}{2}\int\limits_M\sigma^{p-2}|\nabla\sigma|^2\phi^2+\int\limits_M\sigma^{p}|\nabla\phi|^2.
				\end{split}
			\end{equation}
			
			Now we find an estimate for the third term on the right side of (\ref{ineq28sigma}). First, notice that for, any $j$, 
			\begin{equation*}
				\begin{split}
					\left(f_iv_iv_j\sigma^{p-1}\phi^2\right)_j=&\ f_{ij}v_iv_j\sigma^{p-1}\phi^2 +f_iv_{ij}v_j\sigma^{p-1}\phi^2+ f_iv_iv_{jj}\sigma^{p-1}\phi^2\\
					&+(p-1)f_iv_iv_j\sigma^{p-2}\sigma_j\phi^2+2f_iv_iv_j\sigma^{p-1}\phi\phi_j.
				\end{split}
			\end{equation*}
			Thus, by the divergence theorem, we have
			\begin{equation}\label{eq211sigma}
				\begin{split}
					\int\limits_Mf_{ij}v_iv_j\sigma^{p-1}\phi^2=&-\int\limits_Mf_iv_{ij}v_j\sigma^{p-1}\phi^2-\int\limits_Mf_iv_iv_{jj}\sigma^{p-1}\phi^2\\
					&-(p-1)\int\limits_Mf_iv_iv_j\sigma_j\sigma^{p-2}\phi^2-2\int\limits_Mf_iv_iv_j\sigma^{p-1}\phi\phi_j.
				\end{split}
			\end{equation}
			Each term {in the equation} above can be estimated as follows. For the first term, we apply Green's identity getting
			%\begin{equation*}
			%    \begin{split}
				%        -\int\limits_Mf_iv_{ij}v_j\sigma^{p-1}\phi^2=&-\frac{1}{2}\int\limits_M\left\langle\nabla|\nabla v|^2,\nabla f\right\rangle\sigma^{p-1}\phi^2=-\frac{1}{2}\int\limits_M\left\langle\sigma^{p-1}\nabla\sigma,\nabla f\right\rangle\phi^2\\
				%        =&-\frac{1}{2p}\int\limits_M\left\langle\nabla\sigma^{p},\nabla f\right\rangle\phi^2\\
				%        =&-\frac{1}{2p}\left(\int\limits_M\langle\nabla(\sigma^p\phi^2),\nabla f\rangle - \int\limits_M\langle\nabla\phi^2,\nabla f\rangle\sigma^p\right)\\
				%        =&\frac{1}{2p}\int\limits_M(\Delta f)\sigma^p\phi^2+\frac{1}{2p}\int\limits_M\langle\nabla\phi^2,\nabla f\rangle\sigma^p\\
				%        \leq&\frac{1}{2p}\int\limits_M(\Delta f)\sigma^p\phi^2+\frac{1}{p}\int\limits_M|\nabla\phi||\nabla f|\phi\sigma^p
				%    \end{split}
			%\end{equation*}
			%Given $\Delta f=(n\rho-1)R+n\lambda$ is bounded, \textcolor{purple}{for $p$ big enough} inequality above implies
			%\begin{equation*}
			%    \begin{split}
				%        -\int\limits_Mf_iv_{ij}v_j\sigma^{p-1}\phi^2\leq&\int\limits_M|\nabla\phi|^2\sigma^p+\int\limits_M|\nabla f|^2\phi^2\sigma^p\\
				%        \leq& \int\limits_M|\nabla\phi|^2\sigma^p+\int\limits_M(\epsilon f+\delta)\phi^2\sigma^p
				%    \end{split}
			%\end{equation*}
			\begin{equation*}
				\begin{split}
					-\int\limits_Mf_iv_{ij}v_j\sigma^{p-1}\phi^2%={}&-\frac{1}{2}\int\limits_M\left\langle\nabla|\nabla v|^2,\nabla f\right\rangle\sigma^{p-1}\phi^2\\
					%={}&-\frac{1}{2}\int\limits_M\left\langle\sigma^{p-1}\nabla\sigma,\nabla f\right\rangle\phi^2\\
					%={}&-\frac{1}{2p}\int\limits_M\left\langle\nabla\sigma^{p},\nabla f\right\rangle\phi^2\\
					={}&-\frac{1}{2p}\left(\int\limits_M\langle\nabla(\sigma^p\phi^2),\nabla f\rangle - \int\limits_M\langle\nabla\phi^2\nabla f\rangle\sigma^p\right)\\
					={}&\frac{1}{2p}\int\limits_M(\Delta f)\sigma^p\phi^2+\frac{1}{2p}\int\limits_M\langle\nabla\phi^2\nabla f\rangle\sigma^p\\
					\leq{}&\frac{1}{2p}\int\limits_M(\Delta f)\sigma^p\phi^2+\frac{1}{p}\int\limits_M|\nabla\phi||\nabla f|\phi\sigma^p\\
					\leq{}& \frac{1}{2p}\int\limits_M(\Delta f)\sigma^p\phi^2+\int\limits_M|\nabla\phi|^2\sigma^p+\int\limits_M(\epsilon f+\delta)\phi^2\sigma^p.
				\end{split}
			\end{equation*}
			For the second term, we have
			\begin{equation*}
				\begin{split}
					-\int\limits_Mf_iv_iv_{jj}\sigma^{p-1}\phi^2=&-\int\limits_M\langle\nabla v,\nabla f\rangle(\Delta v)\sigma^{p-1}\phi^2=\ a\int\limits_M\langle\nabla v,\nabla f\rangle^2\sigma^{p-1}\phi^2 +\int\limits_M\langle\nabla v,\nabla f\rangle\sigma^{p}\phi^2\\
					\leq&\ a\int\limits_M|\nabla f|^2\sigma^{p}\phi^2 + \int\limits_M|\nabla v||\nabla f|\sigma^{p}\phi^2\\
					\leq&\ a\int\limits_M(\delta f+\epsilon)\sigma^{p}\phi^2+ \int\limits_M\left(\frac{np(a+1)}{2}|\nabla f|^2+\frac{1}{2np(a+1)}|\nabla v|^2\right)\sigma^{p}\phi^2\\
					\leq&\ np(a+1)\int\limits_M(\delta f+\epsilon)\sigma^{p}\phi^2+ \frac{1}{np(a+1)}\int\limits_M\sigma^{p+1}\phi^2.
				\end{split}
			\end{equation*}
			Third term on (\ref{eq211sigma}) can be bounded by
			%\begin{equation*}
			%    \begin{split}
				%       -(p-1)\int\limits_Mf_iv_iv_j\sigma_j\sigma^{p-2}\phi^2\leq&(p-1)\int\limits_M|\nabla f||\nabla v|^2|\nabla\sigma|\sigma^{p-2}\phi^2\leq \int\limits_M\left(p|\nabla f|\sigma\right)|\nabla\sigma|\sigma^{p-2}\phi^2\\
				%       \leq& \int\limits_M\left(\frac{p^28n(a+1)}{2(p-2)}|\nabla f|^2\sigma^2+\frac{p-2}{16n(a+1)}|\nabla\sigma|^2\right)\sigma^{p-2}\phi^2\\
				%       \leq& \frac{p^24n(a+1)}{p-2}\int\limits_M(\delta f +\epsilon)\sigma^p\phi^2+\frac{p-2}{8n(a+1)}\int\limits_M|\nabla\sigma|^2\sigma^{p-2}\phi^2
				%    \end{split}
			%\end{equation*}
			%\color{blue}
			%\begin{equation*}
			%    \begin{split}
				%       -(p-1)\int\limits_Mf_iv_iv_j\sigma_j\sigma^{p-2}\phi^2\leq& -(p-2)\int\limits_Mf_iv_iv_j\sigma_j\sigma^{p-2}\phi^2\leq(p-2)\int\limits_M|\nabla f|\sigma|\nabla\sigma|\sigma^{p-2}\phi^2 \\
				%       \leq& (p-2)\int\limits_M\left({2n(a+1)}|\nabla f|^2\sigma^2 +\frac{|\nabla\sigma|^2}{8n(a+1)}\right)\sigma^{p-2}\phi^2\\
				%       \leq& {2pn(a+1)}\int\limits_M(\delta f +\epsilon)\sigma^p\phi^2+\frac{p-2}{8n(a+1)}\int\limits_M|\nabla\sigma|^2\sigma^{p-2}\phi^2
				%    \end{split}
			%\end{equation*}
			%\color{black}
			\begin{equation*}
				\begin{split}
					-(p-1)\int\limits_Mf_iv_iv_j\sigma_j\sigma^{p-2}\phi^2\leq& p\int\limits_M|\nabla f|\sigma|\nabla\sigma|\sigma^{p-2}\phi^2 \\
					\leq& p\int\limits_M\left({2n(a+1)}|\nabla f|^2\sigma^2 +\frac{|\nabla\sigma|^2}{8n(a+1)}\right)\sigma^{p-2}\phi^2\\
					\leq& {2pn(a+1)}\int\limits_M(\delta f +\epsilon)\sigma^p\phi^2+\frac{p}{8n(a+1)}\int\limits_M|\nabla\sigma|^2\sigma^{p-2}\phi^2.
				\end{split}
			\end{equation*}
			
			Finally, we can estimate the last term on the right-hand side of (\ref{eq211sigma}) by
			$$-2\int\limits_Mf_iv_iv_j\sigma^{p-1}\phi\phi_j\leq 2\int\limits_M\left(\sigma|\nabla f||\nabla\phi|\phi\right)\sigma^{p-1}\leq\int\limits_M(\delta f+\epsilon)\sigma^p\phi^2+\int\limits_M|\nabla\phi|^2\sigma^{p}.$$
			Plugging these four estimates in (\ref{eq211sigma}), we get
			%\begin{equation*}
			%    \begin{split}
				%        \int\limits_Mf_{ij}v_iv_j\sigma^{p-1}\phi^2\leq&{\int\limits_M|\nabla\phi|^2\sigma^p+\int\limits_M(\epsilon f+\delta)\phi^2\sigma^p}+np(a+1)\int\limits_M(\delta f+\epsilon)\sigma^{p}\phi^2\\
				%        &+ \frac{1}{np(a+1)}\int\limits_M\sigma^{p+1}\phi^2+{2pn(a+1)}\int\limits_M(\delta f +\epsilon)\sigma^p\phi^2\\
				%        &+\frac{p}{8n(a+1)}\int\limits_M|\nabla\sigma|^2\sigma^{p-2}\phi^2
				%        +{\int\limits_M(\delta f+\epsilon)\sigma^p\phi^2+\int\limits_M|\nabla\phi|^2\sigma^{p}}
				%    \end{split}
			%\end{equation*}
			\begin{equation*}
				\begin{split}
					\int\limits_Mf_{ij}v_iv_j\sigma^{p-1}\phi^2\leq&{\frac{1}{2p}\int\limits_M(\Delta f)\sigma^p\phi^2+\int\limits_M|\nabla\phi|^2\sigma^p+\int\limits_M(\epsilon f+\delta)\phi^2\sigma^p}+np(a+1)\int\limits_M(\delta f+\epsilon)\sigma^{p}\phi^2\\
					&+ \frac{1}{np(a+1)}\int\limits_M\sigma^{p+1}\phi^2+{2pn(a+1)}\int\limits_M(\delta f +\epsilon)\sigma^p\phi^2\\
					&+\frac{p}{8n(a+1)}\int\limits_M|\nabla\sigma|^2\sigma^{p-2}\phi^2
					+{\int\limits_M(\delta f+\epsilon)\sigma^p\phi^2+\int\limits_M|\nabla\phi|^2\sigma^{p}},
				\end{split}
			\end{equation*}
			which implies
			%\begin{equation}\label{eq212sigma}
			%    \begin{split}
				%        4n(a+1)\int\limits_Mf_{ij}v_iv_j\sigma^{p-1}\phi^2\leq&\ {4n(a+1)\int\limits_M|\nabla\phi|^2\sigma^p+4n(a+1)\int\limits_M(\epsilon f+\delta)\phi^2\sigma^p}\\
				%        &+4n^2(a+1)^2p\int\limits_M(\delta f+\epsilon)\sigma^{p}\phi^2+ \frac{4}{p}\int\limits_M\sigma^{p+1}\phi^2\\
				%        &+{8n^2(a+1)^2p}\int\limits_M(\delta f +\epsilon)\sigma^p\phi^2+\frac{p}{2}\int\limits_M|\nabla\sigma|^2\sigma^{p-2}\phi^2\\
				%        &+{4n(a+1)\int\limits_M(\delta f+\epsilon)\sigma^p\phi^2+4n(a+1)\int\limits_M\sigma^p|\nabla\phi|^2}\\
				%        \leq&\ C_1(n)(a+1)^2p\int\limits_M(\delta f+\epsilon)\sigma^{p}\phi^2+\frac{4}{p}\int\limits_M\sigma^{p+1}\phi^2\\
				%        &+\frac{p}{2}\int\limits_M|\nabla\sigma|^2\sigma^{p-2}\phi^2
				%        +C_2(n)(a+1)\int\limits_M\sigma^p|\nabla\phi|^2,
				%    \end{split}
			%\end{equation}
			\begin{equation}\label{eq212sigma}
				\begin{split}
					4n(a+1)\int\limits_Mf_{ij}v_iv_j\sigma^{p-1}\phi^2\leq&\ \frac{2n(a+1)}{p}\int\limits_M(\Delta f)\sigma^p\phi^2+4n(a+1)\int\limits_M|\nabla\phi|^2\sigma^p\\
					&+4n(a+1)\int\limits_M(\epsilon f+\delta)\phi^2\sigma^p+4n^2(a+1)^2p\int\limits_M(\delta f+\epsilon)\sigma^{p}\phi^2\\
					&+\frac{4}{p}\int\limits_M\sigma^{p+1}\phi^2+{8n^2(a+1)^2p}\int\limits_M(\delta f +\epsilon)\sigma^p\phi^2+\frac{p}{2}\int\limits_M|\nabla\sigma|^2\sigma^{p-2}\phi^2\\
					&+{4n(a+1)\int\limits_M(\delta f+\epsilon)\sigma^p\phi^2+4n(a+1)\int\limits_M\sigma^p|\nabla\phi|^2}\\
					\leq&\ {2(a+1)}\int\limits_M(\Delta f)\sigma^p\phi^2+C_1(n)(a+1)^2p\int\limits_M(\delta f+\epsilon)\sigma^{p}\phi^2\\
					&+\frac{4}{p}\int\limits_M\sigma^{p+1}\phi^2+\frac{p}{2}\int\limits_M|\nabla\sigma|^2\sigma^{p-2}\phi^2
					+C_2(n)(a+1)\int\limits_M\sigma^p|\nabla\phi|^2,
				\end{split}
			\end{equation}
			for $C_1(n)$ and $C_2(n)$ constants depending only on $n$. By plugging (\ref{ineq29sigma}), (\ref{ineq210sigma}) and (\ref{eq212sigma}) into (\ref{ineq28sigma}) we get 
			%\begin{equation*}
			%    \begin{split}
				%        \int\limits_M\sigma^{p+1}\phi^2\leq&\ 4n\int\limits_M\left(a^2(\delta f+\epsilon)-\rho R-\lambda\right)\sigma^p\phi^2 +4n\left(\frac{a^2}{2p}\int\limits_M (\delta f+\epsilon)\sigma^{p} \phi^2+\frac{2}{p}\int\limits_M \sigma^{p+1}\phi^2\right.\\ 
				%        &\left.  +\frac{2}{p}\int\limits_M|\nabla\phi|^2\sigma^{p}\right)+\left( C_1(n)(a+1)^2p\int\limits_M(\delta f+\epsilon)\sigma^{p}\phi^2+\frac{4}{p}\int\limits_M\sigma^{p+1}\phi^2\right.\\
				%        &\left.+\frac{p}{2}\int\limits_M|\nabla\sigma|^2\sigma^{p-2}\phi^2
				%        +C_2(n)(a+1)\int\limits_M\sigma^p|\nabla\phi|^2\right)\\
				%        &+2n\left(-\frac{p}{2}\int\limits_M\sigma^{p-2}|\nabla\sigma|^2\phi^2+\int\limits_M\sigma^{p}|\nabla\phi|^2\right).
				%    \end{split}
			%\end{equation*}
			\begin{equation*}
				\begin{split}
					\int\limits_M\sigma^{p+1}\phi^2\leq&\ 4n\int\limits_M\left(a^2(\delta f+\epsilon)-\rho R-\lambda\right)\sigma^p\phi^2 +4n\left(\frac{a^2}{2p}\int\limits_M (\delta f+\epsilon)\sigma^{p} \phi^2+\frac{2}{p}\int\limits_M \sigma^{p+1}\phi^2\right.\\ 
					&\left.  +\frac{2}{p}\int\limits_M|\nabla\phi|^2\sigma^{p}\right)+\left(2(a+1)\int\limits_M(\Delta f)\sigma^p\phi^2+ C_1(n)(a+1)^2p\int\limits_M(\delta f+\epsilon)\sigma^{p}\phi^2\right.\\
					&\left.+\frac{4}{p}\int\limits_M\sigma^{p+1}\phi^2+\frac{p}{2}\int\limits_M|\nabla\sigma|^2\sigma^{p-2}\phi^2
					+C_2(n)(a+1)\int\limits_M\sigma^p|\nabla\phi|^2\right)\\
					&+2n\left(-\frac{p}{2}\int\limits_M\sigma^{p-2}|\nabla\sigma|^2\phi^2+\int\limits_M\sigma^{p}|\nabla\phi|^2\right).
				\end{split}
			\end{equation*}
			Thus, taking $p$ big enough such that $1-\frac{8n-4}{p}\geq 0$ we have
			%\begin{equation*}
			%    \begin{split}
				%        0\leq&\left(1-\frac{8n}{p}-\frac{4}{p}\right)\int\limits_M\sigma^{p+1}\phi^2\\
				%        \leq&\ 4n\int\limits_M\left(a^2(\delta f+\epsilon)-\rho R-\lambda\right)\sigma^p\phi^2 +{2a^2}\int\limits_M (\delta f+\epsilon)\sigma^{p}\phi^2+ 8\int\limits_M|\nabla\phi|^2\sigma^{p}\\ 
				%        &+ C_1(n)(a+1)^2p\int\limits_M(\delta f+\epsilon)\sigma^{p}\phi^2+\frac{p}{2}\int\limits_M|\nabla\sigma|^2\sigma^{p-2}\phi^2\\
				%        &+C_2(n)(a+1)\int\limits_M\sigma^p|\nabla\phi|^2-2n(p-2)\int\limits_M\sigma^{p-2}|\nabla\sigma|^2\phi^2+2n\int\limits_M\sigma^{p}|\nabla\phi|^2
				%    \end{split}
			%\end{equation*}
			\begin{equation*}
				\begin{split}
					0\leq&\left(1-\frac{8n}{p}-\frac{4}{p}\right)\int\limits_M\sigma^{p+1}\phi^2\\
					\leq&\ 4n\int\limits_M\left(a^2(\delta f+\epsilon)-\rho R-\lambda\right)\sigma^p\phi^2 +{2a^2}\int\limits_M (\delta f+\epsilon)\sigma^{p}\phi^2+ 8\int\limits_M|\nabla\phi|^2\sigma^{p}\\ 
					&+ C_1(n)(a+1)^2p\int\limits_M(\delta f+\epsilon)\sigma^{p}\phi^2+\frac{p}{2}\int\limits_M|\nabla\sigma|^2\sigma^{p-2}\phi^2\\
					&+C_2(n)(a+1)\int\limits_M\sigma^p|\nabla\phi|^2-{np}\int\limits_M\sigma^{p-2}|\nabla\sigma|^2\phi^2+2n\int\limits_M\sigma^{p}|\nabla\phi|^2,
				\end{split}
			\end{equation*}
			which implies
			%\begin{equation*}
			%    \begin{split}
				%        \left(\frac{2np}{2}-\frac{p}{2}\right)\int\limits_M&\sigma^{p-2}|\nabla\sigma|^2\phi^2\leq\ 4n\int\limits_M\left(a^2(\delta f+\epsilon)-\rho R-\lambda\right)\sigma^p\phi^2\\ 
				%        & +{2a^2}\int\limits_M (\delta f+\epsilon)\sigma^{p} \phi^2+8\int\limits_M|\nabla\phi|^2\sigma^{p}+ C_1(n)(a+1)^2p\int\limits_M(\delta f+\epsilon)\sigma^{p}\phi^2\\ 
				%        &+C_2(n)(a+1)\int\limits_M|\nabla\phi|^2\sigma^p+2n\int\limits_M|\nabla\phi|^2\sigma^{p}.\\
				%    \end{split}
			%\end{equation*}
			\begin{equation}\label{auxineqlapf}
				\begin{split}
					\left({np}-\frac{p}{2}\right)\int\limits_M\sigma^{p-2}|\nabla\sigma|^2\phi^2\leq&\ 4n\int\limits_M\left(a^2(\delta f+\epsilon)-\rho R-\lambda\right)\sigma^p\phi^2+2(a+1)\int\limits_M(\Delta f)\sigma^p\phi^2\\ 
					& +{2a^2}\int\limits_M (\delta f+\epsilon)\sigma^{p} \phi^2+8\int\limits_M|\nabla\phi|^2\sigma^{p}+ C_1(n)(a+1)^2p\int\limits_M(\delta f+\epsilon)\sigma^{p}\phi^2\\ 
					&+C_2(n)(a+1)\int\limits_M|\nabla\phi|^2\sigma^p+2n\int\limits_M|\nabla\phi|^2\sigma^{p}.\\
				\end{split}
			\end{equation}
			%Since ${\rho R+\lambda\geq0}$, we conclude
			Notice that taking the trace of the $\rho$-Einstein soliton equation and recalling $0\leq R\leq K$ and $0\leq\rho\leq1/2(n-1)$ we have
			\begin{equation*}
				\begin{split}
					2(a+1)\Delta f-4n(\rho R+\lambda)=2n(a+1)(\rho R+\lambda)-4n(\rho R+\lambda)\leq2n(a-1)C_3
				\end{split}
			\end{equation*}
			for some constant $C_3$ clearly independent of $p$; by plugging this into (\ref{auxineqlapf}) one gets
			\begin{equation*}%\label{auxineqlapf}
				\begin{split}
					\left({np}-\frac{p}{2}\right)\int\limits_M\sigma^{p-2}|\nabla\sigma|^2\phi^2\leq&\ 4na^2\int\limits_M\left(\delta f+\epsilon\right)\sigma^p\phi^2+2n(a-1)C_3\int\limits_M\sigma^p\phi^2\\ 
					& +{2a^2}\int\limits_M (\delta f+\epsilon)\sigma^{p} \phi^2+8\int\limits_M|\nabla\phi|^2\sigma^{p}+ C_1(n)(a+1)^2p\int\limits_M(\delta f+\epsilon)\sigma^{p}\phi^2\\ 
					&+C_2(n)(a+1)\int\limits_M|\nabla\phi|^2\sigma^p+2n\int\limits_M|\nabla\phi|^2\sigma^{p}.\\
				\end{split}
			\end{equation*}
			
			Finally, given $a+1\leq(a+1)^2$ and $a^2\leq(a+1)^2$ we enlarge some terms on the right and merge similar integral terms to conclude
			\begin{equation}\label{finalinequality}
				\begin{split}
					\frac{p}{2}\int\limits_M\sigma^{p-2}|\nabla\sigma|^2\phi^2\leq&\ C(n)(a+1)^2p\int\limits_M\left(\delta f+C\right)\sigma^p\phi^2
					+C(n)(a+1)\int\limits_M\sigma^p|\nabla\phi|^2\\
				\end{split}
			\end{equation}
			for some constants $c(n)$ and $C$ independent of $p$ and for any cut-off function $\phi$ with support on the unit ball $B_{x_0}(1)$.

			Expression (\ref{finalinequality}) is analogous to that found by Munteanu and Wang on their proof of Theorem \ref{thmmw15}, (see equation (\ref{step2ineq})). Furthermore, as $\rho R\geq0$, it is clear that $Ric_f\geq\lambda g$. Thus, we are on the conditions of Lemma 3.2 of \cite{MuWa1} and, by the arguments of the proof of Theorem 2.1 in \cite{MuWa4}, we can conclude 
			$$\sigma(x)\leq Ce^{c(n)r(x)}\int\limits_{B_x(1)}\sigma$$
			for any $x\in E$. Combined with (\ref{ineqlnhexp}) this implies
			\begin{equation}\label{normgradlnhleq}
				|\nabla \ln{h(x)}|\leq\sqrt{\sigma}\leq C_1e^{\left(c(n)\frac{r(x)}{2}-\frac{a}{32}(r(x)-1)^2+\frac{a\alpha}{4}(r(x)+1)\right)}=Ce^{\left(c_2r(x)-\frac{a}{32}r^2(x)\right)}   
			\end{equation}
			for constants $C$ and $c_2$ {that do not depend on} $r(x)$, which is \textbf{Step 3} for the $\rho$-Einstein solitons under our restrictions. 
			
			Now, let $\gamma:[0,\infty)\to M$ be an arbitrary normalized minimizing geodesic on $M$ with $\gamma(0)=x_0$, then, for $x=\gamma(t)\in\gamma$ we have  that $r(x)=t$, and from (\ref{normgradlnhleq})
			$$\int\limits_\gamma|\nabla\ln{h(x)}|dx=\int\limits_0^\infty|\nabla\ln{h(\gamma(t))}|dt\leq C\int\limits_0^\infty e^{\left(c_2t-\frac{a}{32}t^2\right)}dt\leq K_0$$
			for some finite positive constant $K_0$. Therefore, for an arbitrary $\Tilde{x}\in E$ we can take $\gamma$ as being the normalized minimizing geodesic from $x_0$ such that $\Tilde{x}=\gamma(t_0)$ for some $t_0>0$, and from inequality above we have
			\begin{equation*}
				\begin{split} 
					K_0\geq&\int\limits_0^{t_0}|\nabla\ln{h(\gamma(t))}|dt\\
					&\geq\ln{h(\Tilde{x})}-\ln{h(x_0)},
				\end{split}
			\end{equation*}
			and then 
			$$h(\Tilde{x})\leq e^{K_1}$$
			for some constant $K_1$ depending only on $x_0$, by making $\Tilde{x}\to E(\infty)$ we have $h(\Tilde{x})$ is bounded, contradicting (\ref{hconditions}). Thus, $M$ must have only $\varphi$-non-parabolic ends as wanted to prove.
			
		\end{proof}

		\section{A two-parameter inequality}\label{onephinpend}
		
		By Theorem \ref{nphinpendthm}, if $M$ is not connected at infinity, we may consider a $\varphi$-harmonic function satisfying the properties of Proposition \ref{defuvarphi}. The goal of this section is to prove the following key inequality for this function. In the next section we apply this inequality to prove Theorem \ref{thm1endscl}.

		\begin{proposition}\label{mainineqshr}
			Let $(M^n,g,f,\lambda)$ be a shrinking gradient $\rho$-Einstein soliton with $\rho\geq0$ and scalar curvature satisfying $0\leq R\leq K$, for some constant $K>0$. If $a>0$, $b<a$, $\varphi=-af$ and $u$ is a $\varphi$-harmonic function given as in Proposition \ref{defuvarphi}, then
			\begin{equation*}
				\begin{split}
					2\int\limits_M\vert Ric\vert^2\vert\nabla u\vert e^{bf}\leq&-((2b-a)A+1)\int\limits_M R\langle\nabla|\nabla u|,\nabla f\rangle e^{bf} \\
					&-\int\limits_M R\left((Ab+1)\Delta f-(aA+2)(\rho R+\lambda)\right)|\nabla u|e^{bf}\\
					&-(Ab^2+b)\int\limits_M R|\nabla f|^2|\nabla u|e^{bf}-(a+1)A\int\limits_M Ric(\nabla u, \nabla u)R|\nabla u|^{-1}e^{bf},
				\end{split}
			\end{equation*}
			where $A=1-2(n-1)\rho$.% When $A=0$, the inequality above becomes an equality.
		\end{proposition}
		
		In order to prove Proposition \ref{mainineqshr}, we first establish the integrability of several functions in the lemmas to follow. Throughout the proofs, we consider $\varphi$-non-parabolic ends $E_1$ and $E_{2}$ so that $M=E_2\cup E_1$, and $C$ will denote a generic positive constant, independent of indices or radii, which may change from line to line. Recall that $\rho\geq0$.
		
		\begin{lemma}\label{lem1catend}
			Under the conditions of Proposition \ref{mainineqshr}, we have
			$$\int\limits_M|\nabla u|^2e^{2bf}<\infty.$$
		\end{lemma}
		\begin{proof}
			From the last line in \eqref{defuvarphi}, it suffices to prove the inequality for $b>a/2$. Following \cite{MuWa5} (see details in \cite{pLi}), a function $u$ satisfying \eqref{upropcatino15} can be constructed as the limit of a sequence of $\varphi$-harmonic functions $u_i$ defined on geodesic balls $B(x_0,r_i)$ of radii $r_i\rightarrow\infty$ such that
			\begin{equation}\label{upropcatino15}
				\begin{split}
					u_i=0 \mbox{ on } \partial B(x_0,r_i)\cap E_1,\\
					u_i=1 \mbox{ on } \partial B(x_0,r_i)\cap E_2.
				\end{split}
			\end{equation}
			
			On the other hand, multiplying the identity
			\begin{equation*}
				\Delta_{-2bf}e^{-bf}=-\left(b\Delta f +b^2|\nabla f|^2\right)e^{-bf},
			\end{equation*}
			by $\phi\in C_0^\infty(M)$ and integrating by parts we obtain the inequality
			\begin{equation}\label{ptinlem1}
				\int\limits_M\left(b\Delta f +b^2|\nabla f|^2\right)\phi^2e^{2bf}\leq\int\limits_M|\nabla\phi|^2e^{2bf}.
			\end{equation}
			
			Denote by $E_1(x_0,r)=E_1\cap B_{x_0}(r)$, set $r_0>0$ and let 
			$$\psi(x)=\left\{\begin{array}{ll}
				0& \mbox{ on } M\backslash(E_1\backslash E_1(x_0,r_0)),  \\
				r(x)-r_0& \mbox{ on } E_1(x_0,r_{0}+1)\backslash E_1(x_0,r_0),  \\
				1& \mbox{ on } E_1\backslash E_1(x_0,r_0+1).
			\end{array}
			\right.
			$$
			Let $\phi=\psi u_i$ for $u_i$ as given in (\ref{upropcatino15}), then $\phi\in C_0^\infty(M)$ and from inequality (\ref{ptinlem1}) we get
			\begin{equation}\label{auxlemcat10}
				\int\limits_M\left(b\Delta f +b^2|\nabla f|^2\right)\psi^2u_i^2e^{2bf}\leq\int\limits_M|\nabla(\psi u_i)|^2e^{2bf}.
			\end{equation}
			Expanding the right-hand side of \eqref{auxlemcat10}, using Green's identity and that $u_i$ is $\varphi$-harmonic we get
			\begin{equation}
				\begin{split}\label{auxlemcat11}
					\int\limits_M|\nabla(\psi u_i)|^2e^{2bf}=&\int\limits_M u_i^2|\nabla\psi|^2 e^{2bf}+\frac{1}{2}\int\limits_M\langle\nabla u_i^2,\nabla\psi^2\rangle e^{2bf}+\int\limits_M\psi^2|\nabla u_i|^2 e^{2bf}\\
					%=&\int\limits_M u_i^2|\nabla\psi|^2 e^{2bf}-\int\limits_M\psi^2 u_i\Delta_{(-2bf)}u_i e^{2bf}\\
					=&\int\limits_M u_i^2|\nabla\psi|^2 e^{2bf}-\int\limits_M\psi^2 u_i(\Delta u_i+2b\langle\nabla f,\nabla u_i\rangle) e^{2bf}\\
					=&\int\limits_M u_i^2|\nabla\psi|^2 e^{2bf}-\left(b-\frac{a}{2}\right)\int\limits_M\psi^2 \langle\nabla f,\nabla u_i^2\rangle e^{2bf}.
				\end{split}
			\end{equation}
			By noticing
			\begin{equation*}
				\begin{split}
					\int\limits_M\psi^2 \langle\nabla f,\nabla u_i^2\rangle e^{2bf}
					%&=\int\limits_M \langle\nabla(\psi^2 f e^{2bf}),\nabla u_i^2\rangle -\int\limits_M f\langle\nabla(\psi^2e^{2bf}) ,\nabla u_i^2\rangle \\
					%=&-\int\limits_M u_i^2 \Delta(\psi^2 f e^{2bf}) -\int\limits_M \langle\nabla(\psi^2e^{2bf}) ,\nabla (fu_i^2)\rangle+\int\limits_M u_i^2\langle\nabla(\psi^2e^{2bf}) ,\nabla f\rangle\\%%%%%
					%=&-\int\limits_M u_i^2 \left[\Delta(\psi^2f)e^{2bf}+\psi^2f\Delta e^{2bf}+2\langle\nabla(\psi^2f),\nabla e^{2bf}\rangle\right]\\
					%&+\int\limits_M u_i^2f\Delta(\psi^2e^{2bf}) +\int\limits_M u_i^2\langle\nabla\psi^2 ,\nabla f\rangle e^{2bf} +2b\int\limits_M u_i^2\psi^2|\nabla f|^2e^{2bf}\\%%%%%
					%=&-\int\limits_M u_i^2 \left[f\Delta\psi^2+\psi^2\Delta f+2\langle\nabla\psi^2,\nabla f\rangle+\psi^2f(2b\Delta f+4b^2|\nabla f|^2)\right.\\
					%&\left.+4bf\langle\nabla\psi^2,\nabla f\rangle + 4b\psi^2|\nabla f|^2\right]e^{2bf}\\
					%&+\int\limits_M u_i^2f\left[\Delta\psi^2+\psi^22b\Delta f+4b^2\psi^2|\nabla f|^2+4b\langle\nabla\psi^2,\nabla f\rangle\right]e^{2bf} \\
					%&+\int\limits_M u_i^2\langle\nabla\psi^2 ,\nabla f\rangle e^{2bf} +2b\int\limits_M u_i^2\psi^2|\nabla f|^2e^{2bf}\\
					=-\int\limits_M u_i^2 \left[\psi^2\Delta f+\langle\nabla\psi^2,\nabla f\rangle+2b\psi^2|\nabla f|^2\right]e^{2bf},
				\end{split}
			\end{equation*}
			we conclude from (\ref{auxlemcat11}) that
			\begin{equation}\label{auxlemcat12}
				\begin{split}
					\int\limits_M|\nabla(\psi u_i)|^2e^{2bf}=&\int\limits_M u_i^2|\nabla\psi|^2 e^{2bf}+\left(b-\frac{a}{2}\right)\int\limits_M u_i^2\langle\nabla f,\nabla\psi^2\rangle e^{2bf}\\
					&+\left(b-\frac{a}{2}\right)\int\limits_M u_i^2\left(\psi^2\Delta f+2b \psi^2|\nabla f|^2\right) e^{2bf}.\\
					%\leq&\int\limits_M u_i^2|\nabla\psi|^2 e^{2bf}+\left(b-\frac{a}{2}\right)\int\limits_M u_i^2\langle\nabla f,\nabla\psi^2\rangle e^{2bf}\\ 
					%&+\left(b-\frac{a}{2}\right)\int\limits_M u_i^2\psi^2\left[(n\rho-1)R+n\lambda+2b(\epsilon f+\delta)\right] e^{2bf}
				\end{split}
			\end{equation}
			Since both $\nabla\psi$ and $u_i$ are bounded, and $\nabla\psi$ has support only on the annulus $E_1(x_0,r_{0}+1)\backslash E_1(x_0,r_0)$, first and second integral on the right side of (\ref{auxlemcat12}) are finite. Thus, given the convergence of the $u_i$'s, there is a constant $C>0$ independent of $i$, such that
			\begin{equation}\label{auxlemcat13}
				\int\limits_M|\nabla(\psi u_i)|^2e^{2bf}\leq \left(b-\frac{a}{2}\right)\int\limits_M \left(\Delta f+2b|\nabla f|^2\right)u_i^2\psi^2 e^{2bf}+C.
			\end{equation}
			Combining with (\ref{auxlemcat10}), we conclude that
			$$b\int\limits_M\left(\Delta f +b|\nabla f|^2\right)u_i^2\psi^2e^{2bf}\leq \left(b-\frac{a}{2}\right)\int\limits_M \left(\Delta f+2b|\nabla f|^2\right)u_i^2\psi^2 e^{2bf}+C,$$
			this is,
			$$\int\limits_M\left(\frac{a}{2}\Delta f +\left(ab-b^2\right)|\nabla f|^2\right)u_i^2\psi^2e^{2bf}\leq C.$$
			%$$b(b\alpha-2b\epsilon+a\epsilon)\int\limits_M (f-c_2)u_i^2\psi^2 e^{2bf}\leq C,$$
			Applying Proposition \ref{catino} and Theorem \ref{cor33RGES}, inequality above implies
			$$\int\limits_M\left(\frac{a}{2}[(n\rho-1)R+n\lambda] +\left(ab-b^2\right)(\alpha f-\beta)\right)u_i^2\psi^2e^{2bf}\leq C,$$
			since $R$ is bounded, this means there are constants $C_1$, $C_2$ independent of $i$ and depending only on $a$, $b$, $\lambda$ and $K$ such that
			$$\left(ab-b^2\right)\int\limits_M\left(\alpha f+C_1\right)u_i^2\psi^2e^{2bf}\leq C_2,$$
			which up to a translation implies that
			$$\int\limits_{E_1} fu_i^2 e^{2bf}\leq C.$$ 
			%		From (\ref{auxlemcat13}) we know 
			%		$$\int\limits_{E_1}|\nabla u_i|^2e^{2bf}\leq C'\int\limits_{E_1} \left(\delta f+C''\right) e^{2bf}+C.$$
			Combining the last inequality with (\ref{auxlemcat13}), we conclude that
			\begin{equation*}
				\int\limits_{E_1}|\nabla u_i|^2e^{2bf}\leq C,
			\end{equation*}
			which proves the lemma over $E_1$ for $u_i$. Proceeding analogously, similar estimate holds on $E_2$ as well (given it is, at most, also a $\varphi$-non-parabolic end), thus, by making $i\to\infty$, the Lemma follows.
		\end{proof}

		For the next Lemmas, we will consider the cut-off function
		\begin{equation}\label{philemmas3}
			\phi(x)=\left\{\begin{array}{ll}
				1 & \mbox{ on } D(T), \\
				T+1-f(x) & \mbox{ on } D(T+1)\backslash D(T), \\
				0 & \mbox{ on } M\backslash D(T+1),
			\end{array}\right.  
		\end{equation}
		where $D(T)=\{x\in M : f(x)\leq T\}$.
		\begin{lemma}\label{lem2catend}
			Under the conditions of Proposition \ref{mainineqshr}, we have
			$$\int\limits_M|u_{ij}|^2e^{2bf}<\infty.$$
		\end{lemma}
		\begin{proof}
			Consider $\phi$ as given by (\ref{philemmas3}). From the divergence theorem, we have
			%		\begin{equation*}
				%			\begin{split}
					%				0=\int\limits_M\left(u_{ij}u_i\phi^2e^{2bf}\right)_j=&\int\limits_M u_{ijj}u_i\phi^2e^{2bf}+\int\limits_M|u_{ij}|^2\phi^2e^{2bf}+2\int\limits_Mu_{ij}u_i\phi_j\phi e^{2bf}+2b\int\limits_Mu_{ij}u_i\phi^2f_je^{2bf},
					%			\end{split}
				%		\end{equation*}
			%		this is,
			\begin{equation}\label{413catend}
				\begin{split}
					\int\limits_M|u_{ij}|^2\phi^2e^{2bf}=&-\int\limits_M u_{ijj}u_i\phi^2e^{2bf}-2b\int\limits_Mu_{ij}u_i\phi^2f_je^{2bf}-2\int\limits_Mu_{ij}u_i\phi_j\phi e^{2bf}.
				\end{split}
			\end{equation}
			
			By Young's inequality, second term on the right side can be estimated as
			\begin{equation}\label{4135catend}
				\begin{split}
					-2b\int\limits_Mu_{ij}u_i\phi^2f_je^{2bf}%\leq& 2b\int\limits_M|u_{ij}||\nabla u||\nabla f|\phi^2e^{2bf}\\
					%\leq& \frac{1}{4}\int\limits_M|u_{ij}|^2\phi^2e^{2bf}+4b^2\int\limits_M |\nabla u|^2|\nabla f|^2\phi^2e^{2bf}\\
					\leq& \frac{1}{4}\int\limits_M|u_{ij}|^2\phi^2e^{2bf}+2b^2\int\limits_M |\nabla u|^2|\nabla f|^4\phi^2e^{2bf}+2b^2\int\limits_M |\nabla u|^2\phi^2e^{2bf}.\\
					%       \leq& \frac{1}{4}\int\limits_M|u_{ij}|^2\phi^2e^{2bf}+2b^2\int\limits_M |\nabla u|^2|\nabla f|^4e^{2bf}+2b^2\int\limits_M |\nabla u|^2e^{2bf}.
				\end{split}
			\end{equation}
			
			On the other hand, since $|\nabla f|^2\leq(\epsilon f+\delta)$ and $b<a$, there is $y>0$ such that $b+y<a$ and
			\begin{equation*}
				\begin{split}
					\int\limits_M |\nabla u|^2|\nabla f|^4e^{2bf}%\leq& C_1\int\limits_M |\nabla u|^2f^2e^{2bf}+C_2\int\limits_M |\nabla u|^2fe^{2bf}+C_3\int\limits_M |\nabla u|^2e^{2bf}\\
					\leq&\frac{C_1}{y^2}\int\limits_M |\nabla u|^2(yf)^2e^{2bf} +\frac{C_2}{y}\int\limits_M |\nabla u|^2(yf)e^{2bf} +C_3\int\limits_M |\nabla u|^2e^{2bf}\\
					%\leq&\frac{C_1}{y^2}\int\limits_M |\nabla u|^2e^{2yf}e^{2bf} +\frac{C_2}{y}\int\limits_M |\nabla u|^2e^{yf}e^{2bf} +C_3\int\limits_M |\nabla u|^2e^{2bf}\\
					\leq&\frac{C_1}{y^2}\int\limits_M |\nabla u|^2e^{2(y+b)f} +\frac{C_2}{y}\int\limits_M |\nabla u|^2e^{2(y/2+b)f} +C_3\int\limits_M |\nabla u|^2e^{2bf},
				\end{split}
			\end{equation*}
			which, together with Lemma \ref{lem1catend}, implies
			\begin{equation}\label{414catend}
				\int\limits_M |\nabla u|^2|\nabla f|^4e^{2bf}<\infty.
			\end{equation}
			In view of this and (\ref{4135catend}), we conclude
			\begin{equation}\label{aux1lem2catend}
				-2b\int\limits_Mu_{ij}u_i\phi^2f_je^{2bf}\leq\frac{1}{4}\int\limits_M|u_{ij}|^2\phi^2e^{2bf}+C,
			\end{equation}
			where $C$ independent of $T$. Analogously, last term on the right-hand side of (\ref{413catend}) can be estimated by
			\begin{equation}\label{aux2lem2catend}
				\begin{split}
					-2\int\limits_Mu_{ij}u_i\phi_j\phi e^{2bf}%\leq& 2\int\limits_M|u_{ij}||\nabla u||\nabla\phi|\phi e^{2bf}\\
					%\leq&\frac{1}{4}\int\limits_M|u_{ij}|^2\phi^2e^{2bf}+
					%4\int\limits_M|\nabla u|^2|\nabla f|^2 e^{2bf}\\
					\leq&\frac{1}{4}\int\limits_M|u_{ij}|^2\phi^2e^{2bf}+
					2\int\limits_M|\nabla u|^2e^{2bf}+2\int\limits_M|\nabla u|^2|\nabla f|^4 e^{2bf}\\
					\leq&\frac{1}{4}\int\limits_M|u_{ij}|^2\phi^2e^{2bf}+C.
				\end{split}
			\end{equation}
			Plugging (\ref{aux1lem2catend}) and (\ref{aux2lem2catend}) into (\ref{413catend}) we get~
			$$\int\limits_M|u_{ij}|^2\phi^2e^{2bf}\leq-2\int\limits_M u_{ijj}u_i\phi^2e^{2bf}+C.$$
			Given 
			$$\Delta|\nabla u|^2=2u_iu_{ijj}+2|u_{ij}|^2,$$ 
			it follows from the Bochner formula that
			\begin{equation*}
				\begin{split}
					\int\limits_M|u_{ij}|^2\phi^2e^{2bf}%\leq&-2\int\limits_M u_{ijj}u_i\phi^2e^{2bf}+C\\
					=&{\ } 2\int\limits_M |u_{ij}|^2\phi^2e^{2bf}-\int\limits_M \Delta|\nabla u|^2\phi^2e^{2bf}+C\\
					=&{\ } 2\int\limits_M |u_{ij}|^2\phi^2e^{2bf}-2\int\limits_M \left(\langle\nabla\Delta u,\nabla u\rangle +|u_{ij}|^2+\mbox{Ric}(\nabla u,\nabla u)\right)\phi^2e^{2bf}+C\\
					=&-2\int\limits_M\langle\nabla(\Delta u),\nabla u\rangle\phi^2e^{2bf}
					-2\int\limits_M\mbox{Ric}(\nabla u,\nabla u)\phi^2e^{2bf}+C.
				\end{split}
			\end{equation*}
			%Since
			%$$\nabla(\Delta u\phi^2e^{2bf})=\nabla\Delta u\phi^2e^{2bf}+ 2\Delta u\phi\nabla\phi e^{2bf}+2b\Delta u\phi^2\nabla fe^{2bf},$$
			%and 
			%$$\Delta u=-a\langle\nabla u,\nabla f\rangle,$$
			Now, using Green's identity we obtain
			\begin{equation*}
				\begin{split}
					\int\limits_M\langle\nabla(\Delta u),\nabla u\rangle\phi^2e^{2bf}=&\int\limits_M\left\langle\nabla(\Delta u\phi^2e^{2bf}),\nabla u \right\rangle-2b\int\limits_M\Delta u\left\langle\nabla f,\nabla u \right\rangle\phi^2 e^{2bf}-\int\limits_M\Delta u\left\langle\nabla\phi^2,\nabla u \right\rangle e^{2bf}\\
					=&-\int\limits_M(\Delta u)^2\phi^2e^{2bf} -2b\int\limits_M\Delta u\left\langle\nabla f,\nabla u \right\rangle\phi^2 e^{2bf}-\int\limits_M\Delta u\left\langle\nabla\phi^2,\nabla u \right\rangle e^{2bf}\\
					=&-a^2\int\limits_M\langle\nabla u,\nabla f\rangle^2\phi^2e^{2bf} +a\int\limits_M\langle\nabla u,\nabla f\rangle\left\langle\nabla\phi^2,\nabla u \right\rangle e^{2bf} +2ab\int\limits_M\langle\nabla u,\nabla f\rangle^2\phi^2 e^{2bf}\\
					=&(2ab-a^2)\int\limits_M\langle\nabla u,\nabla f\rangle^2\phi^2 e^{2bf}+a\int\limits_M\langle\nabla u,\nabla f\rangle\left\langle\nabla\phi^2,\nabla u \right\rangle e^{2bf}\\
					%\leq&|2ab-a^2|\int\limits_M|\nabla u|^2|\nabla f|^2\phi^2 e^{2bf}+2a\int\limits_M|\nabla u|^2|\nabla f||\nabla\phi|\phi e^{2bf}\\
					\leq&\left(|2ab-a^2|+2a\right)\int\limits_M|\nabla u|^2|\nabla f|^2 e^{2bf},
				\end{split}
			\end{equation*}
			which from Lemma \ref{lem1catend} and (\ref{414catend}) implies 
			\begin{equation}\label{416catend}
				\left| \int\limits_M\langle\nabla(\Delta u),\nabla u\rangle\phi^2e^{2bf}\right|\leq C,
			\end{equation}
			therefore
			\begin{equation}\label{417catend}
				\int\limits_M|u_{ij}|^2\phi^2e^{2bf}\leq2\left|\int\limits_M\mbox{Ric}(\nabla u,\nabla u)\phi^2e^{2bf}\right|+C.
			\end{equation}
			
			From the $\rho$-Einstein soliton equation, the bounds in $R$ and Lemma \ref{lem1catend},
			\begin{equation}\label{418catend}
				\begin{split}
					\left|\int\limits_M\mbox{Ric}(\nabla u,\nabla u)\phi^2e^{2bf}\right| = \left|\int\limits_M\left[(\rho R+\lambda)|\nabla u|^2-f_{ij}u_iu_j\right] \phi^2e^{2bf}\right|%\\
					%\leq&(\rho R+\lambda)\left|\int\limits_M|\nabla u|^2\phi^2e^{2bf}\right|+\left|\int\limits_Mf_{ij}u_iu_j\phi^2e^{2bf}\right|\\
					\leq\left|\int\limits_Mf_{ij}u_iu_j \phi^2e^{2bf}\right|+C.
				\end{split}
			\end{equation}
			However, $\Delta u=-a\langle\nabla f, \nabla u\rangle$, and we can write
			\begin{equation*}
				\begin{split}
					\langle \nabla(\Delta u), \nabla u\rangle=&-a\left\langle \nabla\langle\nabla f, \nabla u\rangle, \nabla u\right\rangle    =-au_{ji}u_if_j-af_{ji}u_iu_j,
				\end{split}
			\end{equation*}
			hence, in view of (\ref{414catend}), (\ref{416catend}) and Lemma \ref{lem1catend},
			\begin{equation*}
				\begin{split}
					\left|\int\limits_Mf_{ij}u_iu_j \phi^2e^{2bf}\right|\leq& \left|\int\limits_Mu_{ij}u_if_j \phi^2e^{2bf}\right|+\frac{1}{a}\left|\int\limits_M\langle \nabla(\Delta u), \nabla u\rangle \phi^2e^{2bf}\right|\\
					%\leq&\left|\int\limits_Mu_{ij}u_if_j \phi^2e^{2bf}\right|+C\\
					%\leq&\frac{1}{4}\int\limits_M|u_{ij}|^2\phi^2e^{2bf}+\int\limits_M|\nabla u|^2|\nabla f|^2\phi^2e^{2bf}+C\\
					\leq&\frac{1}{4}\int\limits_M|u_{ij}|^2\phi^2e^{2bf}+\frac{1}{2}\int\limits_M|\nabla u|^2|\nabla f|^4e^{2bf}+\frac{1}{2}\int\limits_M|\nabla u|^2e^{2bf}+C\\
					\leq&\frac{1}{4}\int\limits_M|u_{ij}|^2\phi^2e^{2bf}+C,
				\end{split}
			\end{equation*}
			together with (\ref{418catend}), this implies
			\begin{equation}\label{4185catend}
				\begin{split}
					\left|\int\limits_M\mbox{Ric}(\nabla u,\nabla u)\phi^2e^{2bf}\right| \leq \frac{1}{4}\int\limits_M|u_{ij}|^2\phi^2e^{2bf}+C.
				\end{split}
			\end{equation}
			By combining (\ref{417catend}) and (\ref{4185catend}) we conclude
			$$\int\limits_M|u_{ij}|^2\phi^2e^{2bf}<\infty,$$
			which concludes the proof.
		\end{proof}
		
		\begin{lemma}\label{lem3catend}
			Under the conditions of Proposition \ref{mainineqshr}, we have
			$$\int\limits_M|\emph{Ric}|^2|\nabla u|^2e^{2bf}<\infty.$$
		\end{lemma}
		\begin{proof}
			Consider $\phi$ as given by (\ref{philemmas3}). From identity \eqref{cat3} and letting $A=1-2(n-1)\rho$, we get from Green's identity and Lemma (\ref{lem1catend}) that
			\begin{equation}\label{419catend}
				\begin{split}
					2\int\limits_M|\mbox{Ric}|^2|\nabla u|^2\phi^2e^{2bf}=& -A\int\limits_M \Delta R|\nabla u|^2\phi^2e^{2bf}+\int\limits_M \langle\nabla R,\nabla f\rangle |\nabla u|^2\phi^2e^{2bf}\\
					&+2\int\limits_M  R(\rho R+\lambda)|\nabla u|^2\phi^2e^{2bf}\\
					\leq&A \int\limits_M \left\langle\nabla R,\nabla\left(|\nabla u|^2\phi^2e^{2bf}\right)\right\rangle +\int\limits_M \langle\nabla R,\nabla f\rangle |\nabla u|^2\phi^2e^{2bf} +C\\
					=&A\int\limits_M \left\langle\nabla R,\nabla|\nabla u|^2\phi^2+\nabla\phi^2|\nabla u|^2+2b\nabla f|\nabla u|^2\phi^2\right\rangle e^{2bf}\\
					&+\int\limits_M \langle\nabla R,\nabla f\rangle |\nabla u|^2\phi^2e^{2bf} +C\\
					=&2A\int\limits_M \left(u_{ij}u_iR_j\right)\phi^2e^{2bf} +A\int\limits_M\left\langle\nabla R,\nabla\phi^2\right\rangle|\nabla u|^2e^{2bf}\\
					&+(2Ab+1)\int\limits_M \left\langle\nabla R,\nabla f\right\rangle|\nabla u|^2\phi^2e^{2bf}+C.
				\end{split}
			\end{equation}
			%from Green's identity
			%\begin{equation*}
			%\begin{split}
			%-\int\limits_M \Delta R|\nabla u|^2\phi^2e^{2bf}=& \int\limits_M \left\langle\nabla R,\nabla\left(|\nabla u|^2\phi^2e^{2bf}\right)\right\rangle\\
			%=&\int\limits_M \left\langle\nabla R,\nabla\left(|\nabla u|^2\right)\right\rangle\phi^2e^{2bf}+\int\limits_M \left\langle\nabla R,\nabla\phi^2\right\rangle|\nabla u|^2e^{2bf}\\
			%&+2b\int\limits_M \left\langle\nabla R,\nabla f\right\rangle|\nabla u|^2\phi^2e^{2bf}\\
			%=&2\int\limits_M \left(u_{ij}u_iR_j\right)\phi^2e^{2bf} +\int\limits_M\left\langle\nabla R,\nabla\phi^2\right\rangle|\nabla u|^2e^{2bf}\\
			%&+2b\int\limits_M \left\langle\nabla R,\nabla f\right\rangle|\nabla u|^2\phi^2e^{2bf},
			%\end{split}
			%\end{equation*}
			If $\rho\neq\frac{1}{2(n-1)}$, we have $A\neq0$ and (\ref{419catend}) can be estimated as follows:
			
			Since $A\nabla R=2\mbox{Ric}(\nabla f)$, we have from Young's inequality that
			\begin{equation*}
				\begin{split}
					2A\left(u_{ij}u_iR_j\right)\leq&4|u_{ij}||\nabla u||\nabla f||\mbox{Ric}|\leq\frac{1}{4}|\mbox{Ric}|^2|\nabla u|^2+16|u_{ij}|^2|\nabla f|^2.
				\end{split}
			\end{equation*}
			Analogously,
			\begin{equation*}
				\begin{split}
					A\left\langle\nabla R,\nabla\phi^2\right\rangle|\nabla u|^2=&4\mbox{Ric}(\nabla f,\nabla \phi)\phi|\nabla u|^2\leq4|\mbox{Ric}||\nabla f|^2|\nabla u|^2\phi\\
					\leq& \frac{1}{4}|\mbox{Ric}|^2|\nabla u|^2\phi^2+16|\nabla u|^2|\nabla f|^2,
				\end{split}
			\end{equation*}
			and
			\begin{equation*}
				\begin{split} %tirar o split?
					(2Ab+1) \left\langle\nabla R,\nabla f\right\rangle|\nabla u|^2=& \frac{2}{A}(2Ab+1) \mbox{Ric}(\nabla f,\nabla f)|\nabla u|^2\\
					\leq& \frac{1}{4}|\mbox{Ric}|^2|\nabla u|^2+\frac{4(2Ab+1)^2}{A^2}|\nabla u|^2|\nabla f|^4.
				\end{split}
			\end{equation*}
			Therefore, we get from (\ref{419catend})
			\begin{equation}\label{421catend}
				\begin{split}
					2\int\limits_M|\mbox{Ric}|^2|\nabla u|^2\phi^2e^{2bf}\leq& \int\limits_M \left(\frac{1}{4}|\mbox{Ric}|^2|\nabla u|^2+16|u_{ij}|^2|\nabla f|^2\right)\phi^2e^{2bf}\\
					&+\int\limits_M\left( \frac{1}{4}|\mbox{Ric}|^2|\nabla u|^2\phi^2+16|\nabla u|^2|\nabla f|^2 \right)e^{2bf}\\
					&+\int\limits_M \left(\frac{1}{4}|\mbox{Ric}|^2|\nabla u|^2+\frac{4(2Ab+1)^2}{A^2}|\nabla u|^2|\nabla f|^4\right)\phi^2e^{2bf}+C,
				\end{split}
			\end{equation}
			or equivalently,
			\begin{equation*}
				\begin{split}
					\frac{5}{4}\int\limits_M|\mbox{Ric}|^2|\nabla u|^2\phi^2e^{2bf}\leq& 16\int\limits_M |u_{ij}|^2|\nabla f|^2\phi^2e^{2bf} +16\int\limits_M|\nabla u|^2|\nabla f|^2e^{2bf}\\
					&+\frac{4(2Ab+1)^2}{A^2}\int\limits_M |\nabla u|^2|\nabla f|^4\phi^2e^{2bf}+C.
				\end{split}
			\end{equation*}
			Recalling Lemmas \ref{lem1catend} and \ref{lem2catend} and equation (\ref{414catend}), we conclude
			\begin{equation*}
				\begin{split}
					\int\limits_M|\mbox{Ric}|^2|\nabla u|^2\phi^2e^{2bf}<\infty,
				\end{split}
			\end{equation*}
			as we wanted to prove. If $A=0$, then (\ref{419catend}) becomes
			$$2\int\limits_M|\mbox{Ric}|^2|\nabla u|^2\phi^2e^{2bf}\leq\int\limits_M \left\langle\nabla R,\nabla f\right\rangle|\nabla u|^2\phi^2e^{2bf}+C,$$
			which can be estimated as
			\begin{equation}
				\begin{split}
					\int\limits_M \left\langle\nabla R,\nabla f\right\rangle|\nabla u|^2\phi^2e^{2bf}=&\int\limits_M \left\langle\nabla\left(R|\nabla u|^2\phi^2e^{2bf}\right),\nabla f\right\rangle-\int\limits_M R\left\langle\nabla\left(|\nabla u|^2\phi^2e^{2bf}\right),\nabla f\right\rangle\\
					=&-\int\limits_M R\Delta f|\nabla u|^2\phi^2e^{2bf}-2\int\limits_M R|\nabla u|\left\langle\nabla|\nabla u|,\nabla f\right\rangle\phi^2e^{2bf}\\
					&-\int\limits_M R|\nabla u|^2\left\langle\nabla\phi^2,\nabla f\right\rangle e^{2bf}-2b\int\limits_M R|\nabla u|^2|\nabla f|^2\phi^2 e^{2bf}.
				\end{split}
			\end{equation}
			Thus, from Proposition \ref{catino} and Young's and Kato's inequalities, we have
			\begin{equation*}
				\begin{split}
					\int\limits_M \left\langle\nabla R,\nabla f\right\rangle|\nabla u|^2\phi^2e^{2bf}\leq&\int\limits_M \left|\left(\frac{n}{2(n-1)}-1\right)R+n\lambda\right|R|\nabla u|^2\phi^2e^{2bf}\\
					+&2\int\limits_M R|\nabla u||\nabla|\nabla u|||\nabla f|\phi^2e^{2bf} + \int\limits_M R|\nabla u|^2|\nabla\phi^2||\nabla f|e^{2bf}+C\\
					%\leq&C_1\int\limits_M |\nabla u|^2\phi^2e^{2bf}+C_2\int\limits_M|\nabla u|^2|\nabla f|^2\phi^2e^{2bf}+\int\limits_M |\nabla|\nabla u||^2\phi^2e^{2bf}\\
					%&+C_3\int\limits_M |\nabla u|^2|\nabla f|^2\phi e^{2bf}+C\\
					\leq&C_1\int\limits_M |\nabla u|^2\phi^2e^{2bf}+C_2\int\limits_M|\nabla u|^2|\nabla f|^2\phi^2e^{2bf}+\int\limits_M |u_{ij}|^2\phi^2e^{2bf}\\
					&+C_4\int\limits_M |\nabla u|^2\phi^2 e^{2bf} +C_5\int\limits_M |\nabla u|^2|\nabla f|^4 e^{2bf}+C.
				\end{split}
			\end{equation*}
			Applying Lemmas \ref{lem1catend}, \ref{lem2catend} and equation (\ref{414catend}), we get
			$$\int\limits_M|\mbox{Ric}|^2|\nabla u|^2\phi^2e^{2bf}\leq\frac{1}{2}\int\limits_M \left\langle\nabla R,\nabla f\right\rangle|\nabla u|^2\phi^2e^{2bf}+C<\infty.$$
			This concludes the proof of the Lemma.
		\end{proof}
		
		As a consequence of the lemmas above, we get the following estimate:
		\begin{lemma}\label{lem4catend}
			Under the conditions of Proposition \ref{mainineqshr}, we have
			$$\int\limits_M|\nabla u|e^{bf}+ \int\limits_M|u_{ij}|e^{bf}+ \int\limits_M|\mbox{\emph{Ric}}||\nabla u|e^{bf}<\infty.$$
		\end{lemma}
		\begin{proof}
			According to Lemmas \ref{lem1catend}, \ref{lem2catend} and \ref{lem3catend}, by making $\overline{b}=\frac{1}{2}(a+b)<a$, we have
			$$\int\limits_M|\nabla u|^2e^{2\overline{b}f}<\infty,\ \ \ \int\limits_M|u_{ij}|^2e^{2\overline{b}f}<\infty \ \ \mbox{ and } \ \ \int\limits_M|\mbox{Ric}|^2|\nabla u|^2e^{2\overline{b}f}<\infty.$$
			%\textcolor{red}{from the Euclidean volume growth of M (??????), we also have }
			%$$\int\limits_Me^{-cf}<\infty$$
			%\textbf{for any c>0.}
			Now, since $b-a<0$, we get from the Cauchy-Schwarz inequality that
			\begin{equation*}
				\left(\int\limits_M|\nabla u|e^{bf}\right)^2= \left(\int\limits_M|\nabla u|e^{\overline{b}f}\cdot e^{(\overline{b}-a)f}\right)^2 \leq\left( \int\limits_M|\nabla u|^2e^{2\overline{b}f} \right)\left( \int\limits_Me^{(b-a)f} \right)<\infty,
			\end{equation*}
			analogously, we can easily check that
			$$\left(\int\limits_M|u_{ij}|e^{bf}\right)^2<\infty \ \mbox{ and }\ \left(\int\limits_M|\mbox{Ric}||\nabla u|e^{bf}\right)^2<\infty,$$
			by taking square roots and adding up these inequalities the Lemma is proven.
		\end{proof}
		
		Finally, we present the proof of Proposition \ref{mainineqshr}.
		
		\begin{proof}[Proof of Propostion \ref{mainineqshr}]
			Consider $\phi$ as given by (\ref{philemmas3}). From identity \eqref{cat3} we have
			\begin{equation}\label{422catend}
				\begin{split}
					2\int\limits_M|\mbox{Ric}|^2|\nabla u|\phi^2e^{bf}=& -A\int\limits_M \Delta R|\nabla u|\phi^2e^{bf}+\int\limits_M \langle\nabla R,\nabla f\rangle |\nabla u|\phi^2e^{bf}\\
					&+2\int\limits_M R(\rho R+\lambda) |\nabla u|\phi^2e^{2bf}.
				\end{split}
			\end{equation}
			From Green's identity, second integral on the right-hand side can be written as
			\begin{equation*}
				\begin{split}
					\int\limits_M \langle\nabla R,\nabla f\rangle |\nabla u|\phi^2e^{bf}=&\int\limits_M \langle\nabla f,\nabla \left(R|\nabla u|\phi^2e^{bf}\right)\rangle-\int\limits_M R\langle\nabla f,\nabla \left(|\nabla u|\phi^2e^{bf}\right)\rangle \\
					=&-\int\limits_M (\Delta f)R|\nabla u|\phi^2e^{bf}-\int\limits_M R\langle\nabla f,\nabla |\nabla u|\rangle\phi^2e^{bf}\\
					&-\int\limits_M R\langle\nabla f,\nabla \phi^2\rangle|\nabla u|e^{bf} -b\int\limits_M R|\nabla f|^2|\nabla u|\phi^2e^{bf},
				\end{split}
			\end{equation*}
			by plugging this into (\ref{422catend}), we get
			\begin{equation}\label{424catend}
				\begin{split}
					2\int\limits_M|\mbox{Ric}|^2|\nabla u|\phi^2e^{bf}=& -A\int\limits_M \Delta R|\nabla u|\phi^2e^{bf}-\int\limits_M (\Delta f)R|\nabla u|\phi^2e^{bf}\\
					&-\int\limits_M R\langle\nabla f,\nabla |\nabla u|\rangle\phi^2e^{bf}-\int\limits_M R\langle\nabla f,\nabla \phi^2\rangle|\nabla u|e^{bf} \\
					&-b\int\limits_M R|\nabla f|^2|\nabla u|\phi^2e^{bf}+2\int\limits_M R(\rho R+\lambda) |\nabla u|\phi^2e^{2bf}.
				\end{split}
			\end{equation}
			%Regarding the first term on the right side of (\ref{422catend}), it can be estimated as follows. 
			On the one hand, from the Bochner formula, $\varphi$-harmonicity of $u$ and the soliton equation, we have
			\begin{align}\label{425catend}
				\begin{split}
					\frac{1}{2}\Delta|\nabla u|^2%=&|u_{ij}|^2+\langle\nabla\Delta u,\nabla u\rangle+\mbox{Ric}(\nabla u, \nabla u)\\
					%=&|u_{ij}|^2-a\left\langle\nabla \langle \nabla u,\nabla f\rangle,\nabla u\right\rangle+\mbox{Ric}(\nabla u, \nabla u)\\
					%=&|u_{ij}|^2-au_{ij}u_if_j-af_{ij}u_iu_j+\mbox{Ric}(\nabla u, \nabla u)\\
					=&|u_{ij}|^2-au_{ij}u_if_j-a(\rho R+\lambda)|\nabla u|^2+a\mbox{Ric}(\nabla u, \nabla u)+\mbox{Ric}(\nabla u, \nabla u)\\
					%=&|u_{ij}|^2-\frac{a}{2}\langle\nabla|\nabla u|^2,\nabla f\rangle -a(\rho R+\lambda)|\nabla u|^2+(a+1)\mbox{Ric}(\nabla u, \nabla u)\\
					=&|u_{ij}|^2-a|\nabla u|\langle\nabla|\nabla u|,\nabla f\rangle -a(\rho R+\lambda)|\nabla u|^2+(a+1)\mbox{Ric}(\nabla u, \nabla u).
				\end{split}
			\end{align}
			On the other hand, Kato's inequality implies
			$$\frac{1}{2}\Delta|\nabla u|^2=|\nabla u|\Delta|\nabla u|+|\nabla|\nabla u||^2\leq|\nabla u|\Delta|\nabla u|+|u_{ij}|^2,$$
			therefore, it follows from (\ref{425catend}) that
			\begin{equation*}
				\Delta|\nabla u|\geq-a\langle\nabla|\nabla u|,\nabla f\rangle -a(\rho R+\lambda)|\nabla u|+(a+1)\mbox{Ric}(\nabla u, \nabla u)|\nabla u|^{-1}.
			\end{equation*}
			Thus,
			\begin{align}\label{426catend}
				\begin{split}
					\Delta\left(|\nabla u|e^{bf}\right)=&\Delta|\nabla u|e^{bf}+|\nabla u|\Delta e^{bf} +2\langle\nabla|\nabla u|,\nabla e^{bf}\rangle\\
					\geq&-a\langle\nabla|\nabla u|,\nabla f\rangle e^{bf} -a(\rho R+\lambda)|\nabla u|e^{bf}+(a+1)\mbox{Ric}(\nabla u, \nabla u)|\nabla u|^{-1}e^{bf}\\
					&+b|\nabla u|\left(\Delta f\right)e^{bf}+b^2|\nabla u||\nabla f|^2e^{bf}+2b\langle\nabla|\nabla u|,\nabla f\rangle e^{bf}\\
					=&(2b-a)\langle\nabla|\nabla u|,\nabla f\rangle e^{bf}+ \left(b(\Delta f)+b^2|\nabla f|^2-a(\rho R+\lambda)\right)|\nabla u|e^{bf}\\
					&+(a+1)\mbox{Ric}(\nabla u, \nabla u)|\nabla u|^{-1}e^{bf}.
				\end{split}
			\end{align}
			From Green's identity,
			\begin{align*}
				-\int\limits_M(\Delta R)|\nabla u|\phi^2e^{bf}%=\int\limits_M\langle \nabla R,\nabla\left(|\nabla u|\phi^2e^{bf}\right)\rangle\\
				%=&\int\limits_M\langle \nabla R,\nabla\left(|\nabla u|e^{bf}\right)\rangle\phi^2+\int\limits_M\langle \nabla R,\nabla\phi^2\rangle|\nabla u|e^{bf}\\
				=&\int\limits_M\langle \nabla (R\phi^2),\nabla\left(|\nabla u|e^{bf}\right)\rangle-\int\limits_MR\langle \nabla \phi^2,\nabla\left(|\nabla u|e^{bf}\right)\rangle        +\int\limits_M\langle \nabla R,\nabla\phi^2\rangle|\nabla u|e^{bf}\\
				=&-\int\limits_M R\Delta\left(|\nabla u|e^{bf}\right)\phi^2-\int\limits_MR\langle \nabla \phi^2,\nabla\left(|\nabla u|e^{bf}\right)\rangle        +\int\limits_M\langle \nabla R,\nabla\phi^2\rangle|\nabla u|e^{bf};
			\end{align*}
			combined with (\ref{426catend}), this becomes
			\begin{align*}
				-\int\limits_M(\Delta R)|\nabla u|\phi^2e^{bf}\leq&-(2b-a)\int\limits_M R\langle\nabla|\nabla u|,\nabla f\rangle e^{bf}\phi^2\\
				&-\int\limits_M R\left(b(\Delta f)+b^2|\nabla f|^2-a(\rho R+\lambda)\right)|\nabla u|e^{bf}\phi^2\\
				&-(a+1)\int\limits_M \mbox{Ric}(\nabla u, \nabla u)R|\nabla u|^{-1}e^{bf}\phi^2\\
				&-\int\limits_MR\langle \nabla \phi^2,\nabla\left(|\nabla u|e^{bf}\right)\rangle+\int\limits_M\langle \nabla R,\nabla\phi^2\rangle|\nabla u|e^{bf}.
			\end{align*}
			By plugging the inequality above into (\ref{424catend}), we get
			\begin{equation*}
				\begin{split}
					2\int\limits_M|\mbox{Ric}|^2|\nabla u|\phi^2e^{bf}\leq& -(2b-a)A\int\limits_M R\langle\nabla|\nabla u|,\nabla f\rangle e^{bf}\phi^2\\
					&-A\int\limits_M R\left(b(\Delta f)+b^2|\nabla f|^2-a(\rho R+\lambda)\right)|\nabla u|e^{bf}\phi^2\\
					&-(a+1)A\int\limits_M \mbox{Ric}(\nabla u, \nabla u)R|\nabla u|^{-1}e^{bf}\phi^2\\
					&-A\int\limits_MR\langle \nabla \phi^2,\nabla\left(|\nabla u|e^{bf}\right)\rangle+A\int\limits_M\langle \nabla R,\nabla\phi^2\rangle|\nabla u|e^{bf}\\
					&-\int\limits_M (\Delta f)R|\nabla u|\phi^2e^{bf}-\int\limits_M R\langle\nabla f,\nabla |\nabla u|\rangle\phi^2e^{bf}\\
					&-\int\limits_M R\langle\nabla f,\nabla \phi^2\rangle|\nabla u|e^{bf} -b\int\limits_M R|\nabla f|^2|\nabla u|\phi^2e^{bf}\\
					&+2\int\limits_M R(\rho R+\lambda) |\nabla u|\phi^2e^{2bf},\\
				\end{split}
			\end{equation*}
			this is,
			\begin{equation}\label{428catend}
				\begin{split}
					2\int\limits_M|\mbox{Ric}|^2|\nabla u|\phi^2e^{bf}\leq&-((2b-a)A+1)\int\limits_M R\langle\nabla|\nabla u|,\nabla f\rangle e^{bf}\phi^2 \\
					&-\int\limits_M R\left((Ab+1)(\Delta f)-aA(\rho R+\lambda)-2(\rho R+\lambda)\right)|\nabla u|e^{bf}\phi^2\\
					&-(Ab^2+b)\int\limits_M R|\nabla f|^2|\nabla u|\phi^2e^{bf}-(a+1)A\int\limits_M \mbox{Ric}(\nabla u, \nabla u)R|\nabla u|^{-1}e^{bf}\phi^2\\
					&-A\int\limits_MR\langle \nabla \phi^2,\nabla\left(|\nabla u|e^{bf}\right)\rangle-\int\limits_M R\langle\nabla f,\nabla \phi^2\rangle|\nabla u|e^{bf}+A\int\limits_M\langle \nabla R,\nabla\phi^2\rangle|\nabla u|e^{bf}.
				\end{split}
			\end{equation}
			
			Regarding the last three terms on inequality above, %we can proceed as in the proof of inequality (\ref{414catend}) to see that there are constants $C_i<\infty$ and $b_j<a$ such that
			%\textcolor{red}{\begin{equation*}
					%    \begin{split}
						%        \int\limits_MR\langle\nabla \phi^2,\nabla\left(|\nabla u|e^{bf}\right)\rangle=&\int\limits_MR\langle\nabla \phi^2,\nabla|\nabla u|\rangle e^{bf}+b\int\limits_MR\langle\nabla \phi^2,\nabla f\rangle|\nabla u|e^{bf}\\
						%        \leq& C_1\int\limits_{D(T+1)\backslash D(T)}|\nabla f||\nabla|\nabla u||e^{bf} +C_2\int\limits_{D(T+1)\backslash D(T)}|\nabla f|^2|\nabla u|e^{bf}\\
						%        \leq&C_3\int\limits_{D(T+1)\backslash D(T)}|\nabla f|^2e^{bf}+C_4\int\limits_{D(T+1)\backslash D(T)}|\nabla u|^2e^{bf}\\
						%        &+C_5\int\limits_{D(T+1)\backslash D(T)}|\nabla f|^2|\nabla u|e^{bf}\\
						%        \leq&C_6\int\limits_{D(T+1)\backslash D(T)}fe^{b_1f}+C_4\int\limits_{D(T+1)\backslash D(T)}|\nabla u|^2e^{bf}\\
						%        &+C_7\int\limits_{D(T+1)\backslash D(T)}f|\nabla u|e^{b_2f}
						%        +C_8\int\limits_{D(T+1)\backslash D(T)}e^{b_2f}
						%    \end{split}
					%\end{equation*}}
					notice in first place that we can take (\ref{414catend}) and proceed like in Lemma \ref{lem4catend} to conclude 
					\begin{equation}\label{finiteintnorm}
						\int\limits_M |\nabla f|^2|\nabla u|e^{bf}<\infty.
					\end{equation}
					Next, from Kato's inequality,
					\begin{equation*}
						\begin{split}
							\int\limits_MR\langle\nabla \phi^2,\nabla\left(|\nabla u|e^{bf}\right)\rangle=&\int\limits_MR\langle\nabla \phi^2,\nabla|\nabla u|\rangle e^{bf}+b\int\limits_MR\langle\nabla \phi^2,\nabla f\rangle|\nabla u|e^{bf}\\
							%						\leq& C_1\int\limits_{D(T+1)\backslash D(T)}|\nabla \phi||\nabla|\nabla u||e^{bf} +C_2\int\limits_{D(T+1)\backslash D(T)}|\nabla f|^2|\nabla u|e^{bf}\\
							\leq& C_1\int\limits_{D(T+1)\backslash D(T)}|\nabla \phi||u_{ij}|e^{bf} +C_2\int\limits_{D(T+1)\backslash D(T)}|\nabla f|^2|\nabla u|e^{bf}.
						\end{split}
					\end{equation*}
					From Lemma \ref{lem4catend} and (\ref{finiteintnorm}), we know both integrals on the right-hand side above vanish as $T\to\infty$, then, we conclude that
					$$-A\int\limits_MR\langle \nabla \phi^2,\nabla\left(|\nabla u|e^{bf}\right)\rangle-\int\limits_M R\langle\nabla f,\nabla \phi^2\rangle|\nabla u|e^{bf}\to0 \ \mbox{ as }\ T\to\infty.$$
					
					Analogously, given $A\nabla R=2\mbox{Ric}(\nabla f)$, we can use Young's inequality to estimate the last term on (\ref{428catend}) by
					\begin{equation*}
						\begin{split}
							A\int\limits_M\langle \nabla R,\nabla\phi^2\rangle|\nabla u|e^{bf}\leq& \int\limits_{D(T+1)\backslash D(T)}|\mbox{Ric}|^2e^{bf}+ \int\limits_{D(T+1)\backslash D(T)}|\nabla f|^4|\nabla u|^2e^{bf},
						\end{split}
					\end{equation*}
					and from Lemma \ref{lem4catend} and (\ref{414catend}) we infer
					$$A\int\limits_M\langle \nabla R,\nabla\phi^2\rangle|\nabla u|e^{bf}\to0 \ \mbox{ as } \ T\to\infty.$$
					In view of this, by making $T\to\infty$ in (\ref{428catend}), we conclude
					\begin{equation}\label{429catend}
						\begin{split}
							2\int\limits_M|\mbox{Ric}|^2|\nabla u|e^{bf}\leq&-((2b-a)A+1)\int\limits_M R\langle\nabla|\nabla u|,\nabla f\rangle e^{bf} \\
							&-\int\limits_M R\left((Ab+1)(\Delta f)-aA(\rho R+\lambda)-2(\rho R+\lambda)\right)|\nabla u|e^{bf}\\
							&-(Ab^2+b)\int\limits_M R|\nabla f|^2|\nabla u|e^{bf}-(a+1)A\int\limits_M \mbox{Ric}(\nabla u, \nabla u)R|\nabla u|^{-1}e^{bf},
						\end{split}
					\end{equation}
					as we wanted to prove.
				\end{proof}
				
				\section{Connectedness at infinity}\label{lastsec}
				
				In this section, we apply Proposition \ref{mainineqshr} to prove the second main theorem of this paper. We shall suppose the existence of at least two ends and use Theorem \ref{nphinpendthm} to guarantee the existence of a $\varphi$-harmonic function satisfying \eqref{defuvarphi}. We then choose the constants $a$ and $b$ suitably to derive a contradiction.
				
				\begin{proof}[Proof of Theorem \ref{thm1endscl}.]
					Assume by contradiction that $M$ has two ends. Then there is a $\varphi$-harmonic function $u$ satisfying (\ref{defuvarphi}) and (\ref{upropcatino15}). {As the scalar curvature is bounded and, according to Theorem \ref{cor33RGES} and Proposition \ref{prop42ineqfd}, the potential function grows quadratically, we can apply Lemma \ref{mainineqshr}.}
					
					%%%%%%%%%%%%%%%%%%%%%%%%%%%%  IMPORTANTE  %%%%%%%%%%%%%%%%%%%%%%%
					%Sobre curvatura escalar não-negativa: 
					% (1) É sabido que a curvatura escalar é não-negativa para Ricci e Schouten solitons shrinking geodesicamente completos. No caso de $\rho$-Einstein sólitons em geral, sabe-se que o mesmo vale para aqueles que possuem, em adição, o campo $\nabla f$ completo.
					
					% (2) Ao invès de assumir $\nabla f$ completo, é melhor assumir $R$ não negativa.
					
					% (3) Acredito que é possível mostrar que $R$ é não-negativo no caso completo quando $\rho<\frac{1}{2(n-1)}$.
					%%%%%%%%%%%%%%%%%%%%%%%%%%%%%%%%%%%%%%%%%%%%%%%%%%%%%%%%%%%%%%%%%%%%
					Set $a=A^{-1}=(1-2(n-1)\rho)^{-1}$ and $b=0$ in (\ref{429catend}) to get
					
					\begin{equation}\label{430catend}
						\begin{split}
							%2\int\limits_M|\mbox{\emph{Ric}}|^2|\nabla u|&e^{bf}\leq-((2b-a)A+1)\int\limits_M R\langle\nabla|\nabla u|,\nabla f\rangle e^{bf} \\
							%&-\int\limits_M R\left((Ab+1)(\Delta f)-(aA+2)(\rho R+\lambda)\right)|\nabla u|e^{bf}\\
							%&-(Ab^2+b)\int\limits_M R|\nabla f|^2|\nabla u|e^{bf}-(a+1)A\int\limits_M \mbox{\emph{Ric}}(\nabla u, \nabla u)R|\nabla u|^{-1}e^{bf}\\
							2\int\limits_M|\mbox{{Ric}}|^2|\nabla u|&\leq -\int\limits_M R\left(\Delta f-3(\rho R+\lambda)\right)|\nabla u|-(1+A)\int\limits_M \mbox{{Ric}}(\nabla u, \nabla u)R|\nabla u|^{-1}.\\
							%&\mbox{for }a=1, b=0\\
							%2\int\limits_M|\mbox{Ric}|^2|\nabla u|&\leq-(-A+1)\int\limits_M R\langle\nabla|\nabla u|,\nabla f\rangle -\int\limits_M R\left(\Delta f-(A+2)(\rho R+\lambda)\right)|\nabla u|\\
							%&-2A\int\limits_M \mbox{Ric}(\nabla u, \nabla u)R|\nabla u|^{-1}
						\end{split}
					\end{equation}
					On the other hand, for any constant $\gamma>0$ 
					\begin{equation*}
						\begin{split}
							-\mbox{Ric}(\nabla u, \nabla u)R|\nabla u|^{-1}=-R_{ij}u_iu_j|\nabla u|^{-1}R
							=-(R_{ij}-\gamma Rg_{ij})u_iu_j|\nabla u|^{-1}R-\gamma|\nabla u|R^2.
						\end{split}
					\end{equation*}
					The first term on the right-hand side of the equation above can be estimated as
					\begin{equation*}
						\begin{split}
							-(R_{ij}-\gamma Rg_{ij})u_iu_j|\nabla u|^{-1}R\leq&|R_{ij}-\gamma Rg_{ij}||\nabla u|R\\
							\leq& |R_{ij}-\gamma Rg_{ij}|^2|\nabla u|+\frac{1}{4}|\nabla u|R^2\\
							=&|\mbox{Ric}|^2+\left(\gamma^2n-2\gamma+\frac{1}{4}\right)|\nabla u|R^2,
						\end{split}
					\end{equation*}
					which implies
					\begin{equation}\label{4315catend}
						-\mbox{Ric}(\nabla u, \nabla u)R|\nabla u|^{-1}\leq|\mbox{Ric}|^2|\nabla u|+\left(\gamma^2n-3\gamma+\frac{1}{4}\right)|\nabla u|R^2.
					\end{equation}
					As $\gamma=\frac{3}{2n}$ is the minimum of $\gamma^2n-3\gamma+\frac{1}{4}$, (\ref{4315catend}) becomes
					\begin{equation}\label{432catend}
						-\mbox{Ric}(\nabla u, \nabla u)R|\nabla u|^{-1}\leq|\mbox{Ric}|^2|\nabla u|+\left(\frac{n-9}{4n}\right)|\nabla u|R^2.
					\end{equation}
					Plugging this into (\ref{430catend}), we get
					\begin{equation*}
						\begin{split}
							2\int\limits_M|\mbox{{Ric}}|^2|\nabla u|\leq& -\int\limits_M R\left(\Delta f-3(\rho R+\lambda)\right)|\nabla u|+(1+A)\int\limits_M|\mbox{Ric}|^2|\nabla u|\\
							&+(1+A)\left(\frac{n-9}{4n}\right)\int\limits_M|\nabla u|R^2,
						\end{split}
					\end{equation*}
					this is,
					\begin{equation*}
						\begin{split}
							(1-A)\int\limits_M|\mbox{{Ric}}|^2|\nabla u|\leq& -\int\limits_M R\left(\Delta f-3(\rho R+\lambda)\right)|\nabla u| +(1+A)\left(\frac{n-9}{4n}\right)\int\limits_M|\nabla u|R^2.
						\end{split}
					\end{equation*}
					Since $\Delta f=(n\rho-1)R+n\lambda,$ and $|\mbox{Ric}|^2\geq \frac{R^2}{n}$, the inequality above implies that
					\begin{equation*}
						\begin{split}
							0\leq& -\int\limits_M R\left((n\rho-1)R+n\lambda-3(\rho R+\lambda)\right)|\nabla u| +\left(\frac{(1+A)(n-9)}{4n}-\frac{1-A}{n}\right)\int\limits_M|\nabla u|R^2\\
							=&-(n-3)\lambda\int\limits_M R|\nabla u| +\left[\frac{(1+A)(n-9)}{4n}-\frac{1-A}{n}+3\rho-(n\rho-1)\right]\int\limits_M|\nabla u|R^2,
						\end{split}
					\end{equation*}
					which can be rewritten as
					\begin{equation*}
						\begin{split}
							(n-3)\lambda\int\limits_M R|\nabla u|\leq&\left[\frac{(1-(n-1)\rho)(n-9)}{2n}-\frac{2(n-1)\rho}{n}+(3-n)\rho+1\right]\int\limits_MR^2|\nabla u|,
						\end{split}
					\end{equation*}
					and consequently
					\begin{equation}\label{intest}
						\frac{2n(n-3)\lambda}{3(n-3)-(3n^2-12n+5)\rho}\int\limits_M R|\nabla u|\leq\int\limits_MR^2|\nabla u|.
					\end{equation}
					It is worth noticing that $3(n-3)-(3n^2-12n+5)\rho>0$ for all $n$ as $0\leq\rho<\frac{1}{2(n-1)}$. Together with the hypothesis over $R$, this implies 
					$$R=\frac{2n(n-3)\lambda}{3(n-3)-(3n^2-12n+5)\rho},$$
					and all inequalities above are in fact equalities, in particular, from equality on Young's inequality, we must have
					$$\left|\mbox{Ric}-\frac{3}{2n}Rg\right|=\frac{1}{2}R,$$
					which implies
					\begin{equation}\label{ric2eq}
						\begin{split}
							\vert Ric\vert^2%=&\left(\frac{n+3}{4n}\right) \frac{4n^2(n-3)^2\lambda^2}{(3(n-3)-(3n^2-12n+5)\rho)^2}\\
							=&\frac{n(n+3)(n-3)^2\lambda^2}{(3(n-3)-(3n^2-12n+5)\rho)^2}.
						\end{split}
					\end{equation}
					
					On the other hand, since $R$ is proven to be constant, $\Delta R=\nabla R=0$, then we have from Proposition \ref{catino} that 
					$$|\mbox{Ric}|^2=\rho R^2+\lambda R,$$
					combined with (\ref{ric2eq}) and the value of $R$, this means
					\begin{equation*}
						\begin{split}
							\frac{n(n+3)(n-3)^2\lambda^2}{(3(n-3)-(3n^2-12n+5)\rho)^2}=&\rho \left(\frac{2n(n-3)\lambda}{3(n-3)-(3n^2-12n+5)\rho}\right)^2\\
							&+\frac{2n(n-3)\lambda^2}{3(n-3)-(3n^2-12n+5)\rho}.
						\end{split}
					\end{equation*}
					Since the numerators are not null, $\lambda>0$ and $n\geq 4$, expression above can be simplified to
					$$(n-3)(n+3)=4n\rho(n-3)+6(n-3)-2(n^2-12n+5)\rho,$$
					which after a couple {of} computations leads to
					\begin{equation}\label{contrho}
						n^2-6n+9=-2\rho(n^2-6n+5).
					\end{equation}
					
					Recalling again $n\geq 4$, we can analyze (\ref{contrho}) for three cases: If $n=4$, then from hypothesis $\rho<\frac{1}{2(n-1)}=\frac{1}{6}$, but making $n=4$ in (\ref{contrho}) one gets $\rho=\frac{1}{6}$, which is impossible. In case $n=5$, a simple substitution into (\ref{contrho}) implies $4=0$, again impossible. Finally, for any $n\geq6$, both $n^2-6n+9$ and $n^2-6n+5$ are positive, an this way (\ref{contrho}) implies 
					$$\rho=-\frac{n^2-6n+5}{2(n^2-6n+9)}<0,$$ 
					contradicting once again the hypothesis of $\rho\in\left[0,\frac{1}{2(n-1)}\right)$.

					In any case, (\ref{contrho}) leads to a contradiction for every $n\geq4$. This means such $\varphi$-harmonic function $u$, satisfying (\ref{429catend}) and (\ref{430catend}), shall not exist. Thus, $M$ must have only one end, and the theorem is proven.
				\end{proof}

			\end{document}